\documentclass[a4paper, 12pt]{amsart} 
\usepackage{amssymb,amsmath,amsfonts,latexsym,color,graphicx,stmaryrd,dsfont}
\usepackage{hyperref} 
\usepackage[font=small,labelfont=bf]{caption}
\numberwithin{equation}{section}
\newcommand{\C}{\mathbf{C}}
 
\newcommand{\tr}{\mathrm{tr}}
\newcommand{\dist}{\mathrm{dist}}
\newcommand{\e}{\mathrm{e}}
\newcommand{\HS}{\mathrm{HS}}

\newcommand{\mO}{\mathcal{O}}

\newtheorem{dref}{Definition}[section] \newtheorem{lemma}[dref]{Lemma}
\newtheorem{theo}[dref]{Theorem} \newtheorem{prop}[dref]{Proposition}
\newtheorem{remark}[dref]{Remark} 
\newtheorem{cor}[dref]{Corollary}

\title{Large bi-diagonal matrices and random perturbations}
\author{Johannes Sj\"ostrand}
\address[Johannes Sj\"ostrand]{IMB, 
  Universit\'e de Bourgogne Franche-Comt\'e, 
  UMR 5584 du CNRS, 
  9, avenue Alain Savary - BP 47870 FR-21078 Dijon Cedex.}
\email{johannes.sjostrand@u-bourgogne.fr}
\author{Martin Vogel}
\address[Martin Vogel]{D\'epartement de Math\'ematiques - UMR 8628 CNRS, B\^atiment 440, Universit\'e Paris-Sud, 
15 Rue du Doyen Georges Poitou, F-91405 Orsay Cedex.}
\email{martin.vogel@math.u-psud.fr}
 \date{}
\begin{document}
\dedicatory{Dedicated to the memory of Yuri Safarov}
\begin{abstract}
This is a first paper by the authors dedicated to the distribution of eigenvalues for 
random perturbations of large bidiagonal Toeplitz matrices. 
  \vskip.5cm
  \par\noindent \textsc{R{\'e}sum{\'e}.} 
Ceci est un premier travail par les auteurs sur la distribution des valeurs propres de 
perturbations al\'eatoires de grandes matrices bidiagonales de Toeplitz.
\end{abstract}
\medskip 
\maketitle
\setcounter{tocdepth}{1}
\tableofcontents

\section{Introduction and main result}\label{int}
\setcounter{equation}{0}
It is well-known that for non-normal operators, as opposed to normal 
operators, the norm of the resolvent can be very large even far away from the spectrum. 
Equivalently, the spectrum of 
such operators can be highly unstable under tiny perturbations. 
Originating from a renewed interest in numerical analysis with the 
works of L.N.~Trefethen and M.~Embree \cite{Tr97,TrEm05}, spectral 
instability of non-self-adjoint operators has become an active subject 
of interest. It is the source of many interesting 
effects, as emphasized by the works of E.B.~Davies, M.~Zworski, J.~Sj\"ostrand 
and many others (cf. \cite{Da97,Da99,NSjZw04,ZwChrist10,DaHa09}). 
\par
It is natural to study the effects of small random perturbations on the 
spectra of non-normal operators. A recent series of works by M.~Hager, 
W.~Bordeaux-Montrieux, J.~Sj\"ostrand and M.~Vogel 
\cite{BM,Ha06,Ha06b,HaSj08,SjAX1002,Vo14,Vo14b} has focused on the 
case of elliptic (pseudo-)differential operators subject to small random 
perturbations. It was shown that for a large class of (pseudo-)differential 
operators one obtains a probabilistic Weyl law for the eigenvalues in the 
bulk of the spectrum. 
\par 
Another important example is the case of non-self-adjoint Toeplitz matrices. 
They can arise for example in models of non-hermitian quantum mechanics, cf \cite{GoKh00,HaNe96}. The spectral theory of such operators has been much 
discussed in the past, cf \cite{Wi94,BoSi99}, and from the point of view of spectral 
instability in \cite{TrEm05}.
\par
The simplest example of a truncated Toeplitz operator is the Jordan block matrix. 
M.~Hager and E.B.~Davies \cite{DaHa09} considered the case of large Jordan block matrices subject to small Gaussian random perturbations and showed that with a sufficiently small coupling constant 
most eigenvalues can be found near a circle, with probability close to $1$, as 
the dimension of the matrix $N$ gets large. Furthermore, they give a probabilistic upper bound of order $\log N$ for the number of eigenvalues in the interior of a circle. 
\par
A recent result by A.~Guionnet, P.~Matched Wood and 
O.~Zeitouni \cite{GuMaZe14} implies that when the coupling constant is 
bounded from above and from below by (different) sufficiently negative powers of $N$, then 
the normalized counting measure of eigenvalues of the randomly perturbed Jordan block  converges weakly in probability to the uniform measure on $S^1$ as the dimension of the 
matrix gets large. 
\par
In \cite{Sj15}, J.~Sj\"ostrand 
obtained a probabilistic circular Weyl law for most of the eigenvalues of large Jordan block matrices 
subject to small random perturbations, and in \cite{SjVo15}, we obtained a precise asymptotic formula  
for the average density of the residual eigenvalues in the interior of a circle, where the result of Davies and Hager yielded a logarithmic upper bound on the number of eigenvalues. The  leading term is given 
by the hyperbolic volume form on the unit disk, independent of the dimension $N$. 
\\
\par
The goal of the present work is to study the spectrum of random perturbations of the
following bidiagonal $N\times N$  Toeplitz matrix:
\begin{itemize}
\item[Case I]
\begin{equation}\label{int.1}
P=P_\mathrm{I}=\begin{pmatrix} 0 &a &0 &.. &.. &0\\
b &0 &a &.. &..&0\\
0 &b &0 &.. &..&0\\
.. &.. &.. &..&..&..\\
0 & ..&.. &..&0 &a\\
0 &0 &.. &.. &b &0 \end{pmatrix}
\end{equation}
\end{itemize}
Originally we also wanted to include
\begin{itemize}
\item[Case II]
\begin{equation}\label{int.2}
P=P_\mathrm{II}=
\begin{pmatrix} 0 &a &b &0 &.. &.. &0\\
0 &0 &a &b &.. &..&0\\
0 &0 &0 &a &.. &..&0\\
.. &.. &.. &.. &.. &..&..\\
.. &.. &.. &.. &.. &a&b\\
.. & ..&.. &.. &..&0 &a\\
..&..&..&.. &..&0&0
\end{pmatrix},
\end{equation}
but we decided to postpone much of the study in this case. 
\end{itemize}
Here $a,\, b\in {\bf C}\setminus \{ 0 \}$ and $N\gg 1$. Identifying ${\bf C}^N$
with $\ell^2([1,N])$, $[1,N]=\{ 1,2,..,N\}$ and also with $\ell^2_{[1,N]}({\bf
  Z})$ (the space of all $u\in \ell^2({\bf Z})$ with support in
$[1,N]$), we have:
\begin{equation}\label{int.3}
P_\mathrm{I}=1_{[1,N]}(a\tau _{-1}+b\tau _1)=1_{[1,N]}(a\e^{iD_x}+b\e^{-iD_x}),
\end{equation}
\begin{equation}\label{int.4}
P_\mathrm{II}=1_{[1,N]}(a\tau _{-1}+b\tau _{-2})=1_{[1,N]}(a\e^{iD_x}+b\e^{2iD_x}),
\end{equation}
where $\tau _ku(j)=u(j-k)$ denotes translation by $k$.
\par The symbols of these operators are by definition,
\begin{equation}\label{int.5}
P_\mathrm{I}(\xi )=a\e^{i\xi }+b\e^{-i\xi },\ P_\mathrm{II}(\xi )=a\e^{i\xi }+b\e^{2i\xi }.
\end{equation}
In this work, we consider the following random perturbation of $P_0=P_\mathrm{I}$
\begin{equation}\label{int.5a}
 P_{\delta} := P_0 + \delta Q_{\omega}, 
 \quad
 Q_{\omega}=(q_{j,k}(\omega))_{1\leq j,k\leq N}, 
\end{equation}
where $0\leq\delta\ll 1 $ and $q_{j,k}(\omega)$ are independent and 
identically distributed complex Gaussian random variables, 
following complex Gaussian law $\mathcal{N}_{\C}(0,1)$. 
\\
The following result shows that, with probability close to $1$, most 
eigenvalues are in a small neighbourhood of the ellipse $E_1=P_\mathrm{I}(S^1)$ 
with focal points $\pm 2\sqrt{ab}$ and major semi-axis of length $|a|+|b|$: 
let $\gamma$ be a segment of $E_1$ and $r>0$, put 
\begin{equation}\label{int.6a}
 \Gamma(r,\gamma) = \{
 z\in\C; ~
 \dist(z,E_1) = \dist(z,\gamma) < r \}.
\end{equation}
\begin{theo}\label{evI2}
Let $P=P_\mathrm{I}$ be the bidiagonal matrix in (\ref{int.1}) where $a,b\in
{\bf C}$ satisfy $0<|b|<|a|$. Let $P_\delta $ be as
in (\ref{int.5a}). Choose $\delta \asymp N^{-\kappa }$, $\kappa >5/2$
 and consider the limit of large $N$. 
Let $\gamma$ be a segment of the ellipse $E_1=P_\mathrm{I}(S^1)$ and let 
$\Gamma=\Gamma (r,\gamma )$ be as in (\ref{int.6a}) with 
$(\ln N)/N\ll r\ll 1$. Let $\delta_0$ be small and fixed.
\par Then with probability 
\begin{equation}\label{evl.41.5}
\geq 1 - \mO(1)\left(\frac{1}{r} + \ln N\right)N^{2\kappa}\e^{-2N^{\delta_0}},
\end{equation}
we have
\begin{equation}\label{evI.41}
\left| \#(\sigma (P_\delta )\cap \Gamma ) -\frac{1}{2\pi
  }\mathrm{vol}_{]0,N]\times S^1} P_\mathrm{I}^{-1}(\Gamma ) \right| \le {\mO}(1)N^{\delta_0} (\frac{1}{r}+\ln N ).
\end{equation}
\end{theo}
\par If we choose $\gamma =E^1$ and view $P_\mathrm{I}$ as a function on $]0,N]\times S^1$, then, since $P_\mathrm{I}^{-1}(\Gamma
)=P_\mathrm{I}^{-1}(E_1)=]0,N]\times S^1$, we have 
$$
\frac{1}{2\pi }\mathrm{vol}_{]0,N]\times S^1} P_\mathrm{I}^{-1}(\Gamma )=N
$$
which is equal to the total number of eigenvalues of $P_\delta $, so
the number of eigenvalues outside of $\Gamma $ is bounded be the right
hand side of (\ref{evI.41}). With $r>0$ fixed but arbitrarily small we
get
\begin{cor}\label{evI3} Under the general assumptions in Theorem \ref{evI2},
let $\Gamma $ be any fixed neighborhood of $E_1$. Then with
probability as in (\ref{evl.41.5}), we have
\begin{equation}\label{evI.42}
\left| \# \left( \sigma (P_\delta )\cap ({\bf C}\setminus \Gamma )
  \right)\right|
\le {\mO}(1)N^{\delta_0} \ln N.
\end{equation}
\end{cor}
Figure \ref{fig1} illustrates the result of Theorem \ref{evI2} by showing 
the eigenvalues of the $N\times N$-matrix in \eqref{int.1}, 
with $N=500$, $a=1+i$ and $b=0.5$, perturbed with a complex Gaussian 
random matrix and coupling constant $\delta=10^{-5}$. The line indicates 
the image of the unit circle $S^1$ under the symbol of the matrix \eqref{int.1}. 
We can see that most eigenvalues are close to an ellipse with only very few 
in the interior. This phenomenon has been observed numerically in 
\cite{TrEm05} (we refer in particular to Figures 3.2 and 3.3 in \cite{TrEm05}). For 
more numerical simulations for more general Toeplitz matrices, we refer 
the reader to \cite[Section 7]{TrEm05}.
\begin{figure}[ht]
 \begin{minipage}[b]{0.49\linewidth}
  \centering
  \includegraphics[width=\textwidth]{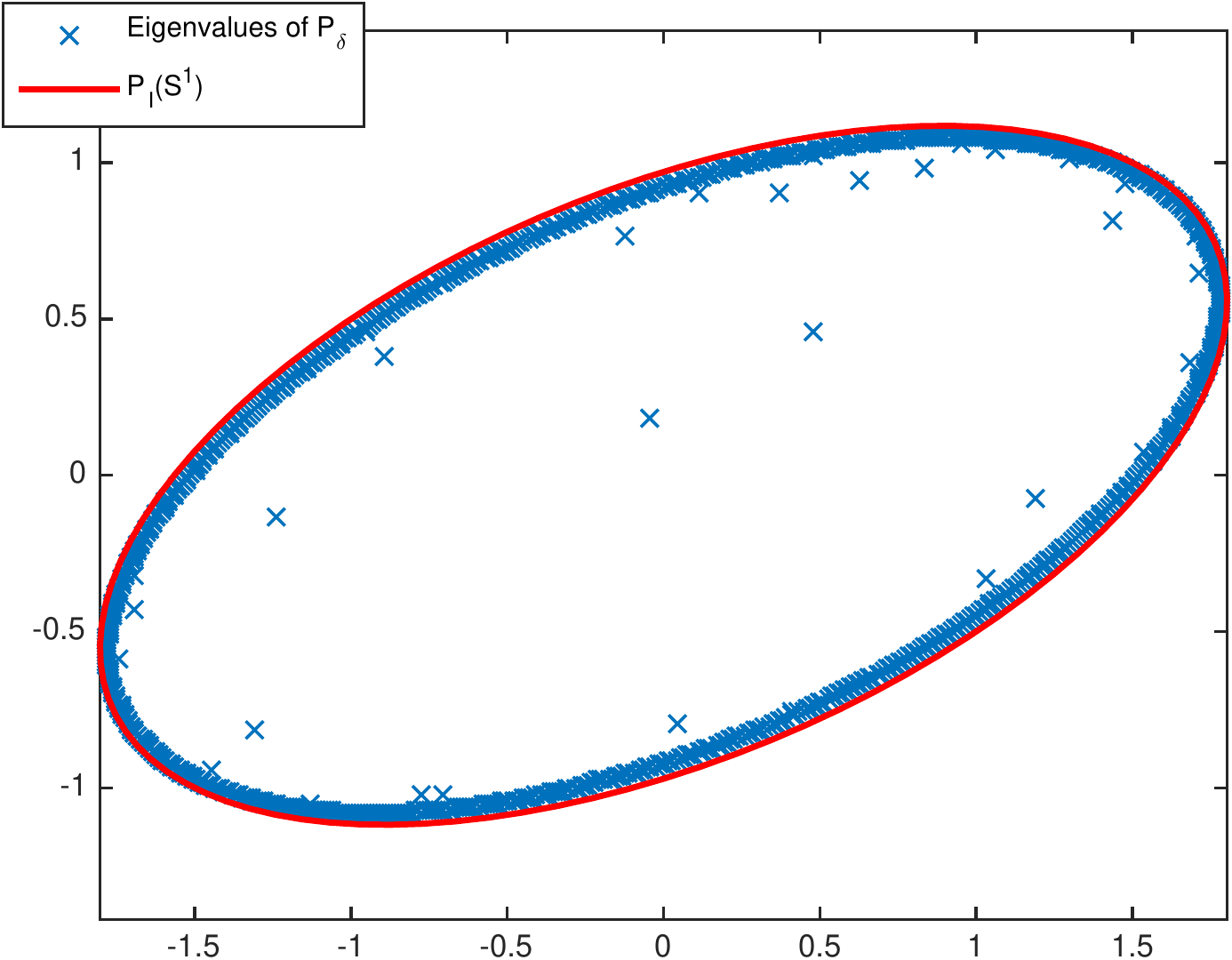}
 \end{minipage}
 \hspace{0cm}
 \begin{minipage}[b]{0.49\linewidth}
  \centering 
  \includegraphics[width=\textwidth]{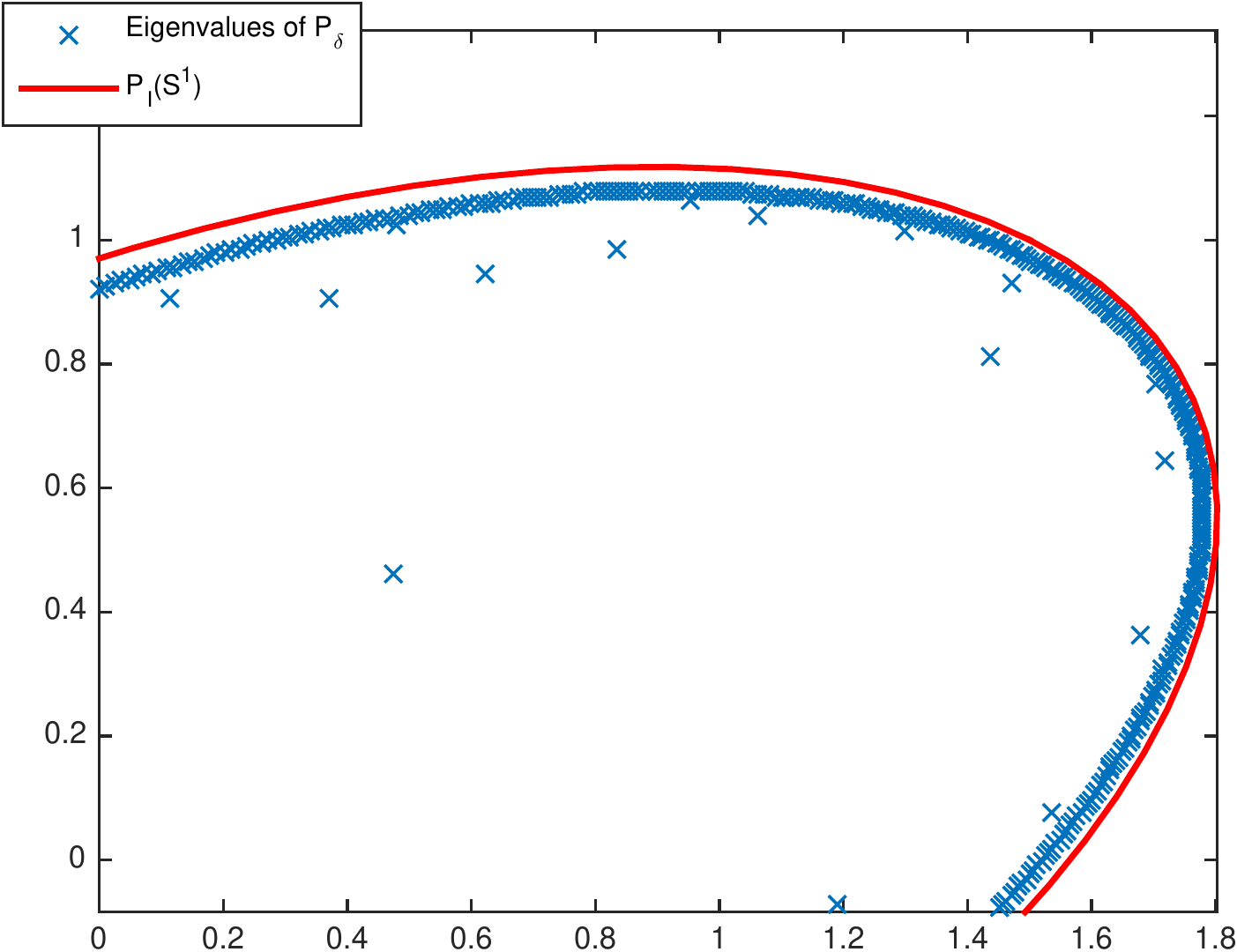}
 \end{minipage}
 \caption{The spectrum of $P_{I}$ with $N=500$, $a=1+i$ and $b=0.5$ perturbed 
 with a complex Gaussian random Matrix with coupling constant $\delta=10^{-5}$. The 
 red line is the image of the unit circle $S^1$ under the symbol $P_\mathrm{I}$.}
  \label{fig1}
\end{figure}
\\
\textbf{Acknowledgements.} J.~Sj\"ostrand was supported by the project 
NOSEVOL ANR 2011 BS 01019 01. M.~Vogel was supported by the projects 
GeRaSic ANR-13-BS01-0007-01 and NOSEVOL ANR 2011 BS 01019 01.%
\section{The range of the symbol}\label{rasy}
\setcounter{equation}{0}
Write
\begin{equation}\label{rasy.1}
a=|a|e^{i\alpha },\  b=|b|e^{i\beta },\ \ \alpha ,\beta \in {\bf R}.
\end{equation}
\subsection{Case I} We have $P(\xi )=|a|e^{i(\alpha +\xi
  )}+|b|e^{i(\beta -\xi )}$ and the largest value of $|P(\xi
)|$ (for $\xi $ real) is attained when the two terms in the expression
for $P(\xi )$ point in the same direction. This happens precisely
when 
$$
\xi =\frac{\beta -\alpha }{2}+\pi k,\ k\in {\bf Z}.
$$
Write $\xi =\frac{\beta -\alpha }{2}+\eta $. Then 
\begin{equation}\label{rasy.2}
\begin{split}P(\xi )&=e^{i\frac{\alpha +\beta }{2}}(|a|e^{i\eta }+|b|
  e^{-i\eta })\\
&=e^{i\frac{\alpha +\beta }{2}}\left( (|a|+|b|)\cos \eta + i(|a|-|b|)\sin
\eta \right).
\end{split}
\end{equation}
Assume, to fix the ideas, that $|b|\le |a|$. Then
$P({\bf R})$ is equal to the ellipse, $E_1$, centred at 0 with major
semi-axis of length $(|a|+|b|)$ pointing in the direction $e^{i(\alpha
+\beta )/2}$ and minor semi-axis of length $|a|-|b|$. The focal points
of $E_1$ are
\begin{equation}\label{rasy.2.5}
\pm 2\sqrt{ab}=\pm e^{i\frac{\alpha +\beta }{2}} 2\sqrt{|a||b|}.
\end{equation}
\par
The left hand side of Figure \ref{fig2} illustrates the range of the symbol in case I by 
presenting $P(S^1)$ with $b=0.5$ and $a=1+i$, $a=0.5+0.5i$.
\subsection{Case II}
Write
$$
P=|a|e^{i(\alpha +\xi )}+|b|e^{i(\beta +2\xi )}.
$$
By the same reasoning as in Case I, the largest value, $|a|+|b|$, of $|P(\xi )|$
is attained when 
$$
\xi =\xi _\mathrm{max}:=\alpha -\beta +2\pi k,\ k\in {\bf Z}.
$$ 
The smallest value $||a|-|b||$ of $|P(\xi )|$ is attained when 
$$
\xi =\xi _\mathrm{min}:=\alpha -\beta +\pi +2\pi k.
$$
We have
\begin{equation}\label{rasy.3}
P(\xi _\mathrm{max})=e^{i(2\alpha -\beta )}(|a|+|b|),\ P(\xi
_\mathrm{min})=-e^{i(2\alpha -\beta )}(|a|-|b|).
\end{equation}

\par Write $\xi =\alpha -\beta +\eta $, so that
\begin{equation}\label{rasy.4}
P(\xi )=e^{i(2\alpha -\beta )}f(e^{i\eta }),\ \ f(\zeta )=|a|\zeta +|b|\zeta ^2.
\end{equation}
The study of $P({\bf R})$ is equivalent to that of $f(S^1)$. Assume
for notational reasons that $a,b>0$, so that 
$$f=f_{a,b}(\zeta )=a\zeta +b\zeta ^2. $$
This function has the unique critical point $\zeta =\zeta _c(a,b)$, given by
$a+2b\zeta _c=0$, 
\begin{equation}\label{rasy.5}
\zeta _c=-\frac{a}{2b}
\end{equation}
and
\begin{equation}\label{rasy.6}
f(\zeta _c)=-\frac{a^2}{4b}.
\end{equation}
Since $f$ is quadratic, we have
\begin{equation}\label{rasy.7}
f(\zeta )=f(\zeta _c)+b(\zeta -\zeta _c)^2=-\frac{a^2}{4b}+b\left( \zeta +\frac{a}{2b}\right) ^2.
\end{equation}
Notice that
\begin{itemize}
\item[1)] $b<a/2$ $\Longrightarrow$ $\zeta _c\not\in \overline{D(0,1)}$,
\item[2)] $b=a/2$ $\Longrightarrow$ $\zeta _c\in S^1$,
\item[3] $b>a/2$ $\Longrightarrow$ $\zeta _c\in D(0,1)$.
\end{itemize}

In the first case there is no pair of distinct points on $S^1$ which
are symmetric to each other with respect to $\zeta _c$ so $f(S^1)$ is a
simple closed curve in ${\bf C}$.

In the second case $f:S^1\to {\bf C}$ is still injective but has a
critical point at $\zeta _c$. The image of $S^1$ is still a simple closed
curve, but with a cusp at $f(\zeta _c)$. 

In the third case, the critical point $\zeta _c$ is situated
on the segment $]-1,0[$. There is one pair of points on $S^1$ that are
symmetric to each other with respect to $\zeta _c$, namely $\alpha _c$ and
$\overline{\alpha _c}$, where $\alpha _c=\zeta _c+i\sqrt{1-\zeta _c^2}$. Thus
$f(\alpha _c)=f(\overline{\alpha _c})\in ]-\infty ,f(\zeta _c)[$ is the
only point in $f(S^1)$ whose inverse image consists of more than one
point. $f(S^1)$ is a closed curve with $f(\alpha
_c)=f(\overline{\alpha _c})$ as its unique
point of self intersection. We can write 
\begin{equation}\label{rasy.8}
f(S^1)=\{f(\alpha _c) \} \cup \gamma _\mathrm{int}\cup \gamma _\mathrm{ext},
\end{equation} 
where $\gamma  _\mathrm{int}=f(S^1\cap \{\zeta ; \Re \zeta <\Re \zeta
_c \} )$, $\gamma
_\mathrm{ext}=f(S^1\cap \{\zeta ; \Re \zeta >\Re \zeta _c \} )$ are smooth curves,
which become simple closed after adding $f(\zeta _c)$ and $\gamma
_\mathrm{int}$ is situated in the interior of the region enclosed by
the closure of $\gamma _\mathrm{ext}$.
\par
The right hand side of Figure \ref{fig2} illustrates the range of the symbol in case II by 
presenting $P(S^1)$ with $b=0.5$ and $a=i$, $a=0.4i$.
\begin{figure}[ht]
 \begin{minipage}[b]{0.53\linewidth}
  \centering
  \includegraphics[width=\textwidth]{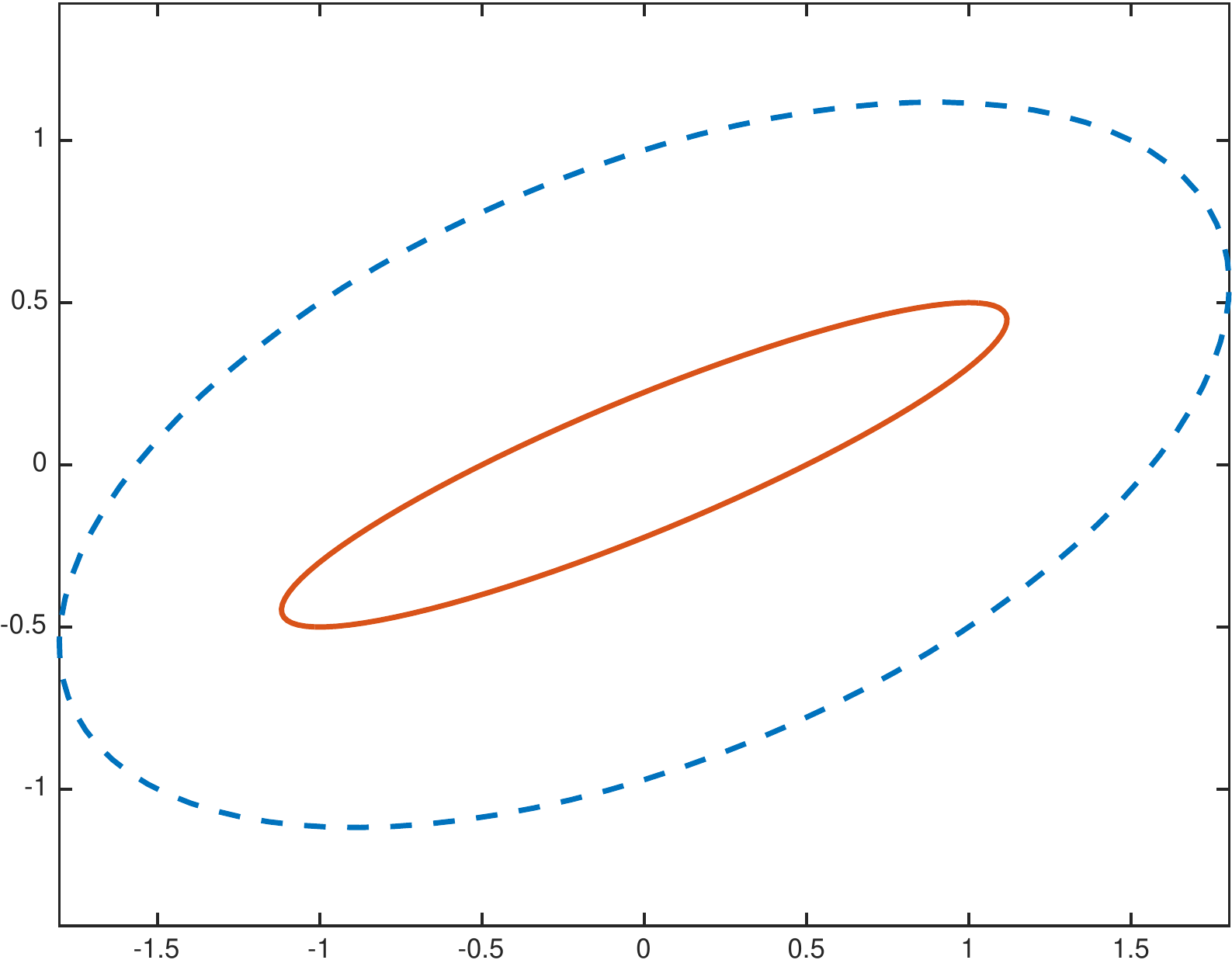}
 \end{minipage}
 \hspace{0cm}
 \begin{minipage}[b]{0.45\linewidth}
  \centering 
  \includegraphics[width=\textwidth]{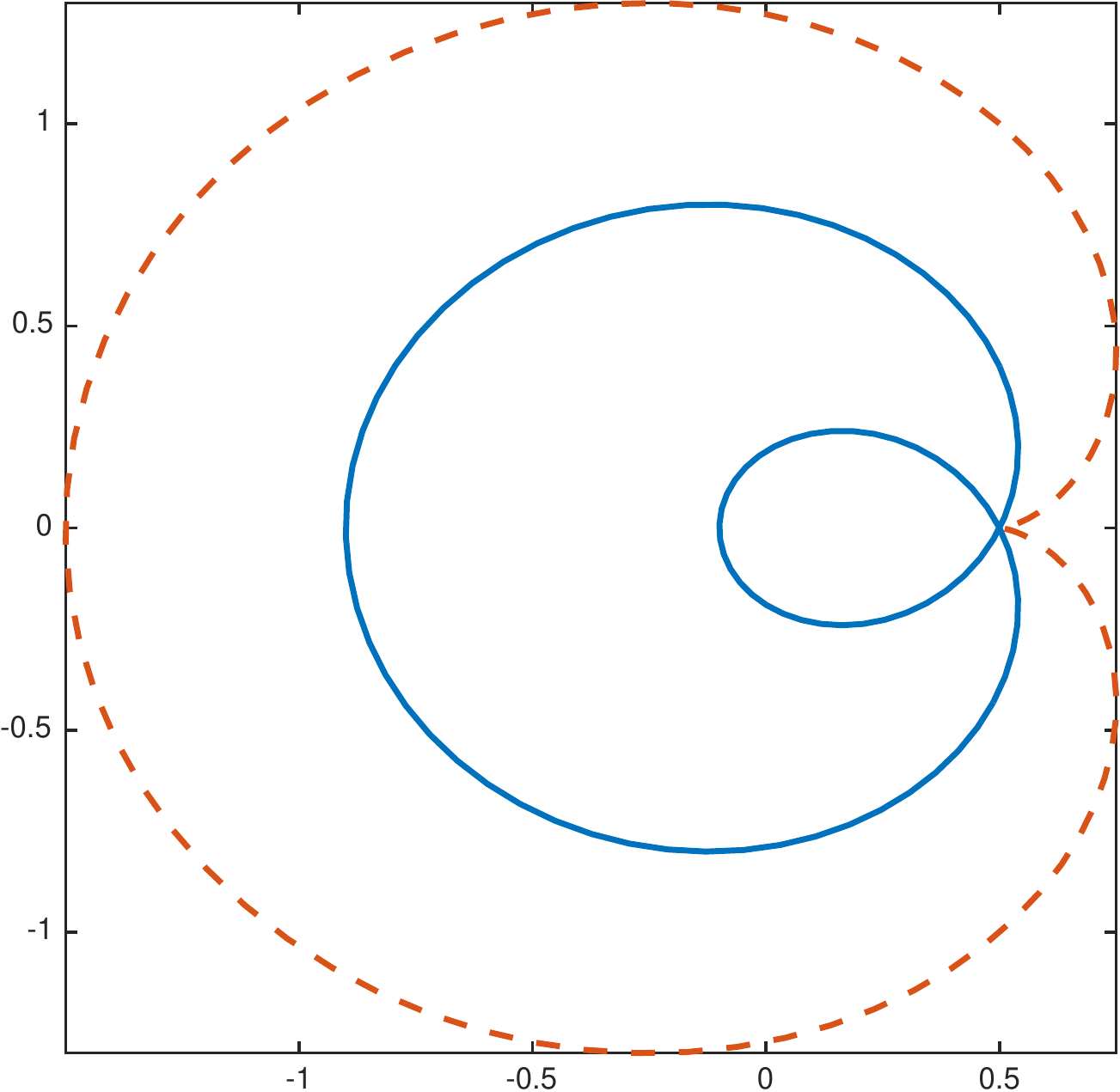}
 \end{minipage}
 \caption{The left hand side shows the image of $S^1$ under the principal symbol of case I (for the dashed ellipse 
 we chose $b=0.5$, $a=1+i$ and for the other ellipse $b=0.5$, $a=0.5+0.5i$). The right hand side  is similar but for the 
 principal symbol of case II (for the dashed line we chose $b=0.5$, $a=i$ and for the continuous line $b=0.5$, $a=0.4i$).}
  \label{fig2}
\end{figure}
\section{Numerical range}\label{nura}
\setcounter{equation}{0}

\noindent\textbf{In case I,} we can write
\begin{equation}\label{nura.1}\begin{split}
&P= 1_{[1,N]}e^{i\frac{\alpha +\beta }{2}}\times \\
&\left((|a|+|b|)\cos
  \left(hD_x-\frac{\beta -\alpha }{2} \right)+i (|a|-|b|)\sin
  \left(hD_x-\frac{\beta -\alpha }{2} \right) \right).
\end{split}
\end{equation} $\cos (hD-(\beta -\alpha )/2)$ and $\sin (hD-(\beta -\alpha )/2)$
are bounded self-adjoint operators on $\ell^2({\bf Z})$, of norm $\le
1$. If $u\in \ell^2([1,N])$ the quantities 
$$
J(u)=( 1_{[1,N]}\cos (hD-(\beta -\alpha )/2)u|u)=(\cos (hD-(\beta -\alpha )/2)u|u)
$$
and
$$
K(u)=( 1_{[1,N]}\sin (hD-(\beta -\alpha )/2)u|u)=(\sin (hD-(\beta -\alpha )/2)u|u)
$$
are real and belong to $[-\| u\|^2,\| u\|^2]$, and we observe that
\begin{equation}\label{nura.2}
(Pu|u)=e^{i\frac{\alpha +\beta }{2}}((|a|+|b|)J(u)+i(|a|-|b|)K(u)).
\end{equation}
By Cauchy-Schwartz, 
$$
J(u)^2\le \| \cos (hD-(\beta -\alpha )/2)u\|^2 \| u\|^2=\left(\cos ^2
(hD-(\beta -\alpha )/2)u|u\right) \| u \|^2,
$$
and similarly,
$$
K(u)^2\le \left(\sin ^2
(hD-(\beta -\alpha )/2)u|u\right)\| u\| ^2.
$$
Summing the two estimates and using that $\cos^2+\sin^2=1$,
it follows that
$$
\sqrt{J(u)^2+K(u)^2}\le \| u\|^2.
$$
If we take $u$ normalised, we deduce from this and (\ref{nura.2})
that
\begin{prop}\label{nura1}
In case I, the numerical range of $P=P_\mathrm{I}$ is contained in the
convex hull of the ellipse $E_1$ described after (\ref{rasy.2}).
\end{prop}

\noindent\textbf{In case II,} we have the trivial inclusion,
$$\mathrm{Num\, Range}(P_\mathrm{II})\subset D(0,|a|+|b|).$$
Using semi-classical analysis and especially the sharp G\aa{}rding
inequality, it seems clear that the numerical range is contained in
a $1/N$-neighborhood of the convex hull of $P(S^1)$.

\section{Spectrum of the unperturbed operator}\label{spnp}
\setcounter{equation}{0}
\subsection{Case I}\label{spnpI} The spectrum of $P_\mathrm{I}$ as a set coincides with the
set of eigenvalues. Consider an eigenvalue $z\in {\bf C}$ and a
corresponding eigenvector $0\ne u\in \ell^2([1,N])$. Extending $u$ to
$[0,N+1]$ by putting $u(0)=0$, $u(N+1)=0$, we have 
\begin{equation}\label{spnp.1}
au(k+1)-zu(k)+bu(k-1)=0,\ k=1,...,N.
\end{equation}
We can extend $u$ further to all of ${\bf Z}$ and get a function
$\widetilde{u}:{\bf Z}\to {\bf C}$ such that 
\begin{equation}\label{spnp.2}
\widetilde{u}(k)=u(k)\hbox{ on }[1,N],\ \widetilde{u}(0)=\widetilde{u}(N+1)=0
\end{equation}
and such that 
\begin{equation}\label{spnp.3}
a\widetilde{u}(k+1)-z\widetilde{u}(k)+b\widetilde{u}(k-1)=0,\ k\in
{\bf Z}.
\end{equation}
The space of solutions to (\ref{spnp.3}) is of dimension 2 and if the
equation 
\begin{equation}\label{spnp.4} a\zeta +b/\zeta -z=0\end{equation}
has two distinct solutions $\zeta _+$ and $\zeta _-$, then it is
generated by the functions $\widetilde{u}_\pm$, given by
\begin{equation}\label{spnp.5}
\widetilde{u}_\pm (k)=\zeta _\pm^k,\ k\in {\bf Z}.
\end{equation}
When the equation has a double solution, which happens precisely when 
\begin{equation}\label{spnp.6}
a\zeta =\frac{b}{\zeta }= \frac{z}{2},
\end{equation}
i.e.\ when $z$ is one of the focal points of $E_1$,
the same space is generated by
\begin{equation}\label{spnp.7}
\widetilde{u}_0(k)= \zeta ^k, \widetilde{u}_1(k)= k \zeta ^k. 
\end{equation}

In the case when the characteristic equation has two distinct
solutions, we can write 
$$
\widetilde{u}(k)=c_+\zeta _+^k+c_-\zeta _-^k,
$$
and apply the boundary conditions in (\ref{spnp.2}), to get
$$
c_++c_-=0,\ \ c_+\zeta _+^{N+1}+c_-\zeta _-^{N+1}=0,
$$
to get 
\begin{equation}\label{spnp.8}
c_-=-c_+,\ \ \zeta _+^{N+1}=\zeta _-^{N+1}.
\end{equation}
This gives the $N$ possibilities,
\begin{equation}\label{spnp.9}
\frac{\zeta _+}{\zeta _-}=e^{2\pi i\nu /(N+1)},\ \nu =1,2,...,N.
\end{equation}
(The case $\nu =0$ is excluded since we are in the case of distinct
solutions of the characteristic equation.) 
The relation between the two solutions of (\ref{spnp.4}) is given by
\begin{equation}\label{spnp.10}
a\zeta _+=b/\zeta _-,
\end{equation}
and insertion of this in (\ref{spnp.9}) gives,
$$
\frac{a}{b}\zeta _+^2=e^{2\pi i\nu /(N+1)}.
$$
Fixing a branch of $\sqrt{b/a}$, we get 
\begin{equation}\label{spnp.11}
\zeta _\pm (\nu ) =\sqrt{b/a}e^{\pm \pi i\nu /(N+1)}, \nu =1,...,N.
\end{equation}
The corresponding eigenvalues are then 
\begin{equation}\label{spnp.12}
  z=z(\nu )=a\zeta _+(\nu )+b/\zeta _+(\nu )=2\sqrt{ab}\cos \left( \frac{\pi \nu }{N+1} \right).
\end{equation}
These values are distinct so we conclude that the spectrum of
$P_\mathrm{I}$ consists of $N$ simple eigenvalues. 

\par Recall the representation (\ref{rasy.1}). We can choose the
branch of the square root so that 
$$
\sqrt{ab}=\sqrt{|a||b|}e^{i(\alpha +\beta )/2}.
$$
We conclude that the eigenvalues
\begin{equation}\label{spnp.13}
z(\nu )=2\sqrt{|a||b|} e^{i(\alpha +\beta )/2}\cos \left( \frac{\pi \nu }{N+1} \right)
\end{equation}
are situated on the major axis of the ellipse $E_1$, between the two focal
points (\ref{rasy.2.5}).
\begin{remark}\label{spnp1}
Let $W=\mathrm{diag\,}(w^k)_{0\le k\le N-1}$ be the diagonal matrix
with elements
$w^k$, $1\le k\le N$, where $0\ne w\in {\bf C}$. Then $P$ has the
same spectrum as 
$$
\widetilde{P}=WPW^{-1}=\begin{pmatrix}0 &a/w &0 &..&.. &0\\
bw &0 &a/w & .. &.. &0\\
0 &bw &0 &a/w &..&0\\
.. &..&..&..&..&..\\
.. &..&..&bw &0 &a/w\\
..&.. &..&..&bw &0
\end{pmatrix},
$$ 
and choosing $w=(a/b)^{1/2}$ gives
$$
\widetilde{P}=(ab)^{1/2}\begin{pmatrix}0 &1 &0 &..&..&.. &0\\
1 &0 &1 & .. &..&.. &0\\
0 &1 &0 &1 &..&..&0\\
.. &..&..&..&..&..&..\\
.. &..&..&..&..&..&..\\
.. &..&.. &..&1 &0 &1\\
..&.. &..&..&..&1 &0
\end{pmatrix},
$$ 
The last matrix is self-adjoint and this explains why the eigenvalues
of $P$ are situated on a segment.
\end{remark}
\subsection{Case II} $P=P_\mathrm{II}$ is nilpotent, so 
\begin{equation}\label{spnp.14}
\sigma (P)=\{ 0\} .
\end{equation}

\section{Size of $|\zeta |$}\label{sizz}
\setcounter{equation}{0}

For the understanding of our operators, it will be important do
determine, depending on $z$, the number of exponential solutions
$\widetilde{u}(k)=\zeta ^k$ that grow near $k=+\infty $
and near $k=-\infty $ respectively. Here $\zeta $ is a solution of the
characteristic equation (\ref{spnp.4}) in case I and of the
characteristic equation
\begin{equation}\label{sizz.1}
a\zeta +b\zeta ^2=z
\end{equation}
in case II. 
\subsection{Case I}\label{sizz.0} 
We recall that we have assumed for simplicity that
$|a|\ge |b|$. The case $|a|=|b|$ will be obtained as a limiting case
of the one when $|a|>|b|$, that we consider now. Let 
$$
f_{a,b}(\zeta )=a\zeta +b/\zeta 
$$
and observe that when $r>0$
$$
f_{a,b}(\partial D(0,r))=f_{ar,b/r}(\partial D(0,1))
$$
which gives a family of confocal ellipses $E_r$. The length of the
major semi-axis of $E_r$ is equal to $|a|r+|b|/r=:g(r)$. $E_{r_1}$ is contained in the bounded domain which has $E_{r_2}$
as its boundary, precisely when $g(r_1)\le g(r_2)$. The function $g$
has a unique minimum at $r=r_\mathrm{min}=(|b|/|a|)^{1/2}$. $g$ is
strictly decreasing on $]0,r_\mathrm{min}]$ and strictly increasing on
$[r_\mathrm{min},+\infty [ $. It tends to $+\infty $ when $r\to 0$ and
when $r\to +\infty $. We have
$g_\mathrm{min}=g(r_\mathrm{min})=2(|a||b|)^{1/2}$ so
$E_{r_\mathrm{min}}$ is just the segment between the two focal points,
common to all the $E_r$. For $r\ne r_\mathrm{min}$, the map $\partial
D(0,r)\to E_r$ is a diffeomorphism. Let $r_1$ be the unique value in
$]0,1[$ for which $g(r_1)=|a|+|b|=g(1)$. We get the following result:
\begin{prop}\label{sizz1} Let $|b|<|a|$. 
\begin{itemize}
\item
When $z$ is strictly inside the ellipse $E_1$ described after
(\ref{rasy.2}), then both solutions of $f_{a,b}(\zeta )=z$ belong to $D(0,1)$.
\item When $z$ is on the ellipse, one solution is on $S^1$ and the
  other belongs to $D(0,1)$.
\item When $z$ is in the exterior region to the ellipse, one solution
  fulfills $|\zeta |>1$ and the other satisfies $|\zeta |<1$. 
\end{itemize}
\end{prop}
\begin{proof}
When $z$ is strictly inside $E_1$ it belongs to $E_{\rho _-}=E_{\rho
  _+},$ $E_{\rho _\pm}=f(\partial D(0,\rho
_\pm))$, for $r_1<\rho _-\le
\rho _+<1$ and $\zeta _\pm \in \partial D(0,\rho _\pm )$. The other
two cases are treated similarly.
\end{proof}

\par In the case $|a|=|b|$, $E_1$ is just the segment between the two
focal points. In this case $r_\mathrm{min}=1$ and we get:
\begin{prop}\label{sizz2}
Assume that $|a|=|b|$. \begin{itemize} 
\item If $z\in E_1$ then both solutions of
$f_{a,b}(\zeta )=z$ belong to $S^1$.
\item If $z$ is outside $E_1$, one solution is in $D(0,1)$ and the
  other is in the complement of $\overline{D(0,1)}$.
\end{itemize}
\end{prop}
\begin{remark}\label{sizz2.5}
Recall that $E_\rho $ is the ellipse with focal points $\pm
c=\pm 2\sqrt{ab}$ and length of major semi-axis equal to
$g(\rho ):=a\rho +b/\rho $. Equivalently,
$$
E_\rho =\{ z\in {\bf C};\, |z-c|+|z+c|=2g(\rho ) \} .
$$
The solutions $\zeta =\zeta _\pm$ of
$f_{a,b}(\zeta )=z$ belong to $\partial D(0,\rho _\pm )$,
where $\rho _\pm$ are the solutions to $2g(\rho _\pm
)=|z-c|+|z+c|$, situated on each side of $r
_\mathrm{min}=\sqrt{b/a}$. Also,
$g_\mathrm{min}=g(r _\mathrm{min})=|c|$.

\par Assume that $z$ is restricted to a compact subset $K$ of ${\bf
  C}$. Then $\rho _\pm \in [1/C,C]$ for some $C=C(K)>1$. For $\rho \in
[1/C,C]$, we have
$$
g(\rho )-|c|\asymp |\rho -r _\mathrm{min}|^2.
$$
Consequently,
$$
|\rho_\pm -r_\mathrm{min}|\asymp (g(\rho _\pm )-|c|)^{\frac{1}{2}}=(|z-c|+|z+c|-2|c|)^{\frac{1}{2}}.
$$
Since $\rho _\pm$ are situated on opposite sides of $r_\mathrm{min}$;
we get
\begin{equation}\label{sizz.1.5}
|\zeta _+-\zeta _-|\ge |\rho _+-\rho _-|\asymp (|z-c|+|z+c|-2|c|)^{\frac{1}{2}}
\end{equation}
\end{remark}

\subsection{Case II}\label{sizz.2.0}
\textcolor{blue}{Can be removed from this paper}
We write the equation (\ref{sizz.1}) as 
\begin{equation}\label{sizz.2}
f(\zeta )=z,\hbox{ where }f(\zeta ):=a\zeta +b\zeta ^2=b(\zeta -\zeta _c)^2+c=z,
\end{equation}
$\zeta _c$ is the critical point, given by $\zeta _c=-a/(2b)$
and 
$$
c=f(\zeta _c)=-a^2/(4b).
$$
For any given $z$, the two solutions are symmetric around $\zeta _c$,
and depending on in which case we are according to the conclusion
after (\ref{rasy.7}), we just have to see if both, one or none of the
symmetric solutions belong to $D(0,1)$. 

\par When $|b|<|a|/2$, we have $\zeta _c\not\in \overline{D(0,1)}$,
$f(S^1)$ is a simple smooth closed curve $\gamma $ and if we let
$\Gamma $ be the bounded open set with $\partial \Gamma =\gamma $,
we conclude:
\begin{prop}\label{sizz3}
\begin{itemize}
\item If $z\in \Gamma $, then one of the solutions belongs to $D(0,1)$ and
the other one is in ${\bf C}\setminus \overline{D(0,1)}$. 
\item If $z\in \gamma $, then one of the solutions belongs to $S^1$ and
  the other is in ${\bf C}\setminus \overline{D(0,1)}$.
\item If $z\in {\bf C}\setminus \overline{\Gamma }$, then both
  solutions are in $z\in {\bf C}\setminus \overline{D(0,1) }$.
\end{itemize}
\end{prop} 

\par When $|b|=|a|/2$, we have:
\begin{prop}\label{sizz4}
For $z\ne f(\zeta _c)$ (i.e.\ for $z$ away from the cusp of the simple
closed curve $f(S^1)$) we have the same conclusion as in Proposition
\ref{sizz3}. When $z=f(\zeta _c)$ (i.e.\ at the cusp), $\zeta =\zeta
_c\in S^1$ is a double solution of $z=f(\zeta )$ (and there is no other).
\end{prop}

\par When $|b|>|a|/2$, we have $\zeta _c\in D(0,1)$ and $\zeta _c\ne
0$. Draw the line through $\zeta _c$ which is perpendicular to the
radius of $D(0,1)$ that passes through that point and let $\alpha _c$
and $\alpha '_c$ be the two points of intersection with $S^1$. We have
seen in Section \ref{rasy} that $f(\alpha _c)=f(\alpha '_c)$ and that
the short circle arcs and the long circle arcs connecting these two
points are mapped by $f$ onto two simple closed curves $\gamma _\mathrm{int}$
and $\gamma _\mathrm{ext}$, both containing the
point $f(\alpha _c)=f(\alpha '_c)$ and such that if we let $\Gamma
_\mathrm{int}$, $\Gamma _\mathrm{ext}$ denote the open bounded sets
bounded by $\gamma _\mathrm{int}$, $\gamma _\mathrm{ext}$
respectively, then away from $f(\alpha _c)$, $\gamma _\mathrm{int}$ is
contained in $\Gamma _\mathrm{ext}$. Also $\Gamma _\mathrm{int}\subset
\Gamma _\mathrm{ext}$.

\par It is geometrically clear that the set of points $\zeta $ in
$\overline{D(0,1)}$ for which the symmetric point with respect to
$\zeta _c$, namely $\zeta '=\zeta _c-(\zeta -\zeta _c)$, also belongs
to $\overline{D(0,1)}$, is obtained by taking the short circular
segment from $\alpha _c$ to $\alpha '_c$, then the convex hull of
that set and finally adding all the symmetric points with respect to
$\zeta _c$. This is a lens shaped region $L$ inside
$\overline{D(0,1)}$ whose image under $f$ coincides with
$\overline{\Gamma } _\mathrm{int}$. This leads to:
\begin{prop}\label{sizz5}
When $|b|>|a|/2$, the following holds for the solutions
(counted with their multiplicity) of the equation $f(\zeta )=z$:
\begin{itemize}
\item If $z\in \Gamma _\mathrm{int}$, we have two solutions in
  $D(0,1)$
\item If $z\in \gamma _\mathrm{int}\setminus \{ f(\alpha _c)\}$ one of the
  solutions is on $S^1$, namely on the short circular arc between
  $\alpha _c$ and $\alpha _c'$, while the other solution is in
  $D(0,1)$.
\item If $z=f(\alpha _c)$, there are two solutions on $S^1$, namely
  $\alpha _c$ and $\alpha '_c$.
\item If
  $z\in \Gamma _\mathrm{ext}\setminus \overline{\Gamma }
  _\mathrm{int}$,
  then there is one solution in $D(0,1)$ and one outside
  $\overline{D(0,1)}$.
\item If $z\in \gamma _\mathrm{ext}\setminus \{ f(\alpha _c)\} $, one
  of the solutions is in $S^1$ (namely on the long circular arc from
  $\alpha _c$ to $\alpha _c'$) and the other one is outside
  $\overline{D(0,1)}$.
\item If $z\in {\bf C}\setminus \overline{\Gamma }_\mathrm{ext}$, both
  solutions are outside $\overline{D(0,1)}$.
\end{itemize}
\end{prop}

\section{Grushin problems for the unperturbed operators}\label{grunp}
\setcounter{equation}{0}
From now on we only consider the case I and write $P=P_\mathrm{I}$. 
We are interested in the case when $z$ is inside
the ellipse $E_1$, so that the two solutions of the characteristic
equation are in $D(0,1)$ and correspond to exponential solutions
that decay in the direction of increasing $k$. Consequently, in our
Grushin problem we put a condition of type ``$+$'' at the endpoint
$k=1$ of $[1,N]$ and a corresponding co-condition of type ``$-$'' at
the right end point $k=N$. Define $R_+:{\bf C}^N\to {\bf C}$ and
$R_-:{\bf C}\to {\bf C}^N$, by 
\begin{equation}\label{grunp.1}
R_+u=au(1),\ \ R_-u_-=au_-e_N,
\end{equation}
for $u\in {\bf C}^N$, $u_-\in {\bf C}$, where $e_N$ denotes the $N$th
canonical basis vector in ${\bf C}^N$ so that $e_N(j)=\delta
_{j,N}$. We are then interested in inverting 
\begin{equation}\label{grunp.2}
\begin{pmatrix}P-z &R_-\\R_+
  &0\end{pmatrix}\begin{pmatrix}u\\u_-\end{pmatrix}
= \begin{pmatrix}v\\ v_+\end{pmatrix}
\end{equation} 
in ${\bf C}^N\times {\bf C}$.

\par It will be convenient (without changing the mathematics) to
permute the components $v,v_+$, so
we apply the block matrix 
$$
\begin{pmatrix}0 &1\\ 1 &0 \end{pmatrix}: {\bf C}^N\times {\bf C}\to
{\bf C}\times {\bf C}^N
$$  
to the left and get the equivalent problem.
\begin{equation}\label{grunp.3}
{\mathcal{ P}}(z)\begin{pmatrix}u\\ u_-\end{pmatrix}= \begin{pmatrix}v_+ \\
  v\end{pmatrix},\hbox{ where }{\mathcal{P}(z)}=\begin{pmatrix}R_+ &0\\ P-z
  &R_-\end{pmatrix}:{\bf C}^N\times {\bf C}\to {\bf C}\times {\bf C}^N.
\end{equation}
If
$$
\begin{pmatrix} E  &E_+\\ E_- &E_{-+}\end{pmatrix}
$$
denotes the inverse of the matrix in (\ref{grunp.2}) (when it exists), then the inverse of the matrix in (\ref{grunp.3}) is
given by 
\begin{equation}\label{grunp.4}
{\mathcal{E}}=\begin{pmatrix}E_+ &E\\E_{-+} &E_-\end{pmatrix}:{\bf C}\times
{\bf C}^N\to {\bf C}^N\times {\bf C}.
\end{equation}
The important quantity $E_{-+}$ now appears in the lower left
corner. 

\par Identifying ${\bf C}^N\times {\bf C}\simeq {\bf C}\times {\bf
  C}^N\simeq {\bf C}^{N+1}$ in the natural way we see that ${\mathcal{P}}(z)$ 
  has a lower triangular matrix with 
\begin{itemize}
\item all entries on the main diagonal equal to $a$,
\item all entries on the ``subdiagonal'' (i.e.\ with indices
  $j,k$ satisfying $j-k=1$) equal to $-z$,
\item all entries on the ``subsubdiagonal'' (with indices satisfying
  $j-k=2$) equal to $b$,
\item all other entries equal to zero.
\end{itemize}
Equivalently, we can write 
\begin{equation}\label{grunp.5}
{\mathcal{P}}(z)=1_{[1,N+1]}(a-z\tau +b\tau ^2),
\end{equation}
where $\tau $ denotes translation to the right by one unit, when
identifying
${\bf C}^{N+1}\simeq \ell^2_{[1,N+1]}({\bf Z})=\{ u\in \ell^2({\bf
  Z}); u(k)=0\hbox{ for }k\not\in [1,N+1] \}$.

\par We see that ${\mathcal{P}}(z)$ is invertible with inverse ${\mathcal{E}}(z)$ 
given by a lower triangular matrix with constant entries
$c_\nu $ on the $\nu $th subdiagonal (i.e.\ the entries with index
$(j,k)$ for which $j-k=\nu $). Further $c_0=1/a$. Equivalently,
\begin{equation}\label{grunp.6}
{\mathcal{E}}(z)=1_{[1,N+1]}(c_0+c_1\tau +...+c _N\tau ^N).
\end{equation}
Also notice that $c_\nu $ is independent of $N$. The
first column $u$ in the matrix of ${\mathcal{E}}(z)$ is equal to
$(c_0,c_1,...,c_N)^\mathrm{t}$ and it solves the problem 
$$
{\mathcal{ P}}(z)u=e_1, \hbox{ where }e_1(j)=\delta _{1,j}.
$$
This gives the equations,
\[
\begin{split}
au(1)&=1,\\
-zu(1)+au(2)&=0,\\
bu(1)-zu(2)+au(3)&=0,\\
bu(2)-zu(3)+au(4)&=0,\\
.&..\\
bu(N-1)-zu(N)+au(N+1)&=0.
\end{split}
\]
Extend $u$ to $u\in \ell^2([0,N+1])$, by putting $u(0)=0$. Then
we get,
\begin{equation}\label{grunp.7}
\begin{cases}bu(k-1)-zu(k)+au(k+1)=0,\ k=1,...,N,\\
u(0)=0,\\
u(1)=1/a.
\end{cases}
\end{equation}
Here $u$ can be extended uniquely to all of ${\bf Z}$ by solving successively
the first equation in (\ref{grunp.7}) for $k=0,-1,..$ and for
$k=N+1,N+2,...$ and the extended function $u$ has to be of the
form
\begin{equation}\label{grunp.8}
u(k)=c_+\zeta _+^k+c_-\zeta _-^k,
\end{equation}
where $\zeta _\pm$ are the solutions of (\ref{spnp.4}), and we assume
that $z$ is not a focal point of $E_1$, so that $\zeta _+\ne \zeta _-$. 
The last two equations in (\ref{grunp.7}) give
$$
c_++c_-=0,\ \ \zeta _+c_++\zeta _-c_-=1/a ,
$$
and we conclude that
\begin{equation}\label{grunp.9}
c_+=\frac{1}{a(\zeta _+-\zeta _-)},\ \ c_-=-\frac{1}{a(\zeta _+-\zeta _-)}
\end{equation}
By (\ref{grunp.4}), we know that $E_{-+}$ is the last component of
$u$, and hence
\begin{equation}\label{grunp.10}
E_{-+}(z)=\frac{\zeta _+^{N+1}-\zeta _-^{N+1}}{a(\zeta _+-\zeta _-)}.
\end{equation}

\par Recall that by a general identity for Grushin problems,
\begin{equation}\label{grunp.11}
(-1)^N E_{-+}(z)\det {\mathcal{P}}(z) =\det (P-z).
\end{equation}
Here the factor $(-1)^N$ comes from the displacement of the 1st line
when going from the ``standard'' matrix in (\ref{grunp.2}) to ${\mathcal{P}}(z)$.
Now,
\begin{equation}\label{grunp.12}
\det {\mathcal{P}}(z)=a^{N+1}.
\end{equation}
The last three equations give,
\begin{equation}\label{grunp.13}
\det (P-z)=(-a)^N\,\frac{\zeta _+^{N+1}-\zeta _-^{N+1}}{\zeta _+-\zeta _-}.
\end{equation}

\par We observe a symmetry property of (\ref{grunp.13}): Reversing the
orientation and permuting $a$ and $b$, we could replace $R_+$ by
$\widetilde{R}_+u=bu(N)$ and $R_-$ by $\widetilde{R}_-u=bu_-e_1$. We
should then replace $\zeta _\pm$ by $1/\zeta _\pm$ and
(\ref{grunp.13}) becomes
\begin{equation}\label{grunp.14}
\det (P-z)=(-b)^N\frac{\zeta _+^{-(N+1)}-\zeta _-^{-(N+1)}}{\zeta
  _+^{-1}-\zeta _-^{-1}}.
\end{equation}
Using that $\zeta _+\zeta _-=b/a$, we check directly that the right
hand sides in (\ref{grunp.13}) and (\ref{grunp.14}) are equal.

\section{Estimates on the Grushin problems and the resolvent}
The aim of this section is to obtain estimates on the 
Grushin problem and the resolvent for the unperturbed operator. 
In the following we will work 
with the convention that the two solutions $\zeta_{\pm}$ 
to the characteristic equations are such that 
\begin{equation*}
 |\zeta_+| \leq |\zeta_-|.
\end{equation*}
\subsection{When $\zeta _+$ and $\zeta _-$ both belong to $D(0,1)$.}
Here we give estimates on \eqref{grunp.4} which 
is the inverse of \eqref{grunp.3}, the Grushin problem for the unperturbed 
operator $P_\mathrm{I}$ in the case when $z$ is inside the ellipse $E_1$. 
\\
\par
By \eqref{grunp.4} and \eqref{grunp.6}
\begin{equation}\label{estgrunp.1}
 E = 1_{[1,N]}(c_0+c_1\tau + \dots + c_{N-2}\tau^{N-2})1_{[2,N+1]},
\end{equation}
\begin{equation}\label{estgrunp.2}
 E_+ = \begin{pmatrix}
        c_0 \\
        \vdots \\
        c_{N-1}
       \end{pmatrix}
  , \quad 
  E_-=(c_{N-1},\dots,c_0),
\end{equation}
where, using \eqref{grunp.8} and \eqref{grunp.9},
\begin{equation}\label{estgrunp.3}
 c_k = \frac{\zeta_+^{k+1}-\zeta_-^{k+1}}{a(\zeta_+-\zeta_-)}, 
 \quad 
 k=0,\dots, N.
\end{equation}
For $k\in\mathds{N}$, $t\in \overline{D(0,1)}$, let 
\begin{equation}\label{estgrunp.3.1}
 F_{k+1}(t) = 1+ t+ \dots + t^k 
 = 
 \begin{cases}
  k+1, ~\text{when } t =1, \\
  \frac{1-t^{k+1}}{1-t}, ~\text{when } t\neq 1. 
 \end{cases}
\end{equation}
By the triangle inequality 
\begin{equation*}
 | F_{k+1}(t) | \leq F_{k+1}(|t|). 
\end{equation*}
We have 
\begin{equation*}
 |F_{k+1}(t)| \leq \min\!\left(k+1, \frac{2}{|1-t|}\right).
\end{equation*}
From \eqref{estgrunp.3} we get  
\begin{equation}\label{estgrunp.4}
 c_k = \frac{\zeta_-^k}{a}F_{k+1}\left(\frac{\zeta_+}{\zeta_-}\right),
\end{equation}
\begin{equation}\label{estgrunp.5}
 |c_k| \leq \frac{|\zeta_{-}|^{k}}{|a|}
 \min\left(k+1,\frac{2}{|1-\zeta_+/\zeta_-|}\right).
\end{equation}
By \eqref{grunp.10}, \eqref{estgrunp.4}, \eqref{estgrunp.5} we have that 
\begin{equation}\label{estgrunp.8}
 |E_{-+}| = \frac{|\zeta_{-}|^{N}}{|a|}|F_{N+1}(\zeta_+/\zeta_-)| 
 \leq \frac{|\zeta_{-}|^{N}}{|a|}
 \min\left(N+1,\frac{2}{|1-\zeta_+/\zeta_-|}\right).
\end{equation}
\begin{prop}
 If $\zeta_{\pm}\in D(0,1)$, then 
\begin{equation}\label{estgrunp.6}
 \| E\| \leq |a|^{-1}
  \min\left(N,\frac{2}{1-|\zeta_-|}\right)
  \min\left(N,\frac{2}{1-|\zeta_-|},\frac{2}{|1-\zeta_+/\zeta_-|}\right),
\end{equation}
\begin{equation}\label{estgrunp.7}
 \|E_{+}\|=\|E_-\| \leq |a|^{-1}
  \min\left(N,\frac{2}{1-|\zeta_-|}\right)^{\frac{1}{2}}
  \min\left(N,\frac{2}{1-|\zeta_-|},\frac{2}{|1-\zeta_+/\zeta_-|}\right).
\end{equation}
%
%
%
\end{prop}
\begin{proof}
 From \ref{estgrunp.1}, we infer that 
 \begin{equation*}
  \|E\| \leq |c_0| + \dots + |c_{N-2}|. 
 \end{equation*}
Then, by \eqref{estgrunp.5}, 
\begin{equation*}
 \| E\| \leq \frac{1}{|a|}
 \min\left(\sum_0^{N-2} |\zeta_-|^k(k+1), 
 \frac{2}{|1-\zeta_+/\zeta_-|}F_{N-1}(|\zeta_-|)\right).
\end{equation*}
Here 
\begin{equation*}
 \sum_0^{N-2}|\zeta_-|^k(k+1) \leq 
 \min\left( (N-1)^2, \sum_0^{\infty}|\zeta_-|^k(k+1)\right)
\end{equation*}
and
\begin{equation*}
 \sum_0^{\infty}|\zeta_-|^k(k+1) = 
 \partial_t \left(\sum_0^{\infty}t^{k+1}\right)_{t=|\zeta_-|}
 = \partial_t\left(\frac{t}{1-t}\right)_{t=|\zeta_-|} = 
 \frac{1}{(1-|\zeta_-|)^2},
\end{equation*}
leading to 
\begin{equation*}
  \| E\| \leq 
  \frac{1}{|a|}
  \min \left( 
  (N-1)^2, 
  \frac{1}{(1-|\zeta_-|)^2}, 
  \frac{2(N-1)}{|1-\zeta_+/\zeta_-|}, 
  \frac{4}{|1-\zeta_+/\zeta_-|(1-|\zeta_-|)}
  \right) 
\end{equation*}
which implies \eqref{estgrunp.6}. Continuing, we see 
by \eqref{estgrunp.2}, (\ref{estgrunp.5}), that 
\begin{equation*}
\begin{split}
 \|E_-\|^2=\|E_+\|^2 &= |c_0|^2 + \dots + |c_{N-1}|^2 \\
 &\leq
 \frac{1}{|a|^2}\min \left(
 \sum_0^{N-1}|\zeta_-|^{2k}(k+1)^2, 
 \frac{4}{|1-\zeta_+/\zeta_-|^2}\sum_0^{N-1}|\zeta_-|^{2k}
 \right).
\end{split}
\end{equation*}
Here, 
\begin{equation*}
 \sum_0^{N-1}|\zeta_-|^{2k} = F_N(|\zeta_-|^2),
\end{equation*}
and for $0\leq t \leq 1$, 
\begin{equation*}
 F_N(t^2) = \frac{1-t^{2N}}{1-t^2} 
 =F_N(t)\frac{1+t^N}{1+t} \leq F_N(t)
\end{equation*}
so 
\begin{equation*}
 \sum_0^{N-1}|\zeta_-|^{2k} \leq F_N(|\zeta_-|).
\end{equation*}
Furthermore, 
\begin{equation*}
 \sum_0^{N-1}|\zeta_-|^{2k}(k+1)^2\leq 
 \min\left( 
 N^3,
 \sum_0^{\infty}|\zeta_-|^{2k}(k+1)^2
 \right).
\end{equation*}
Here, 
\begin{equation*}
 \begin{split}
  \sum_0^{\infty}t^{k}(k+1)^2 
  &= \partial_t t \partial_t\sum_0^{\infty}t^{k+1}
  =\partial_t t \partial_t \frac{t}{1-t} \\
  & = \partial_t t \partial_t \frac{1}{1-t}
  =\partial_t \frac{t}{(1-t)^2}
  =\frac{1+t}{(1-t)^3}.
 \end{split}
\end{equation*}
Hence
\begin{equation*}
 \begin{split}
 \sum_0^{N-1}|\zeta_-|^{2k}(k+1)^2 
 &\leq 
 \min\left(N^3, \frac{1+|\zeta_-|^2}{(1-|\zeta_-|^2)^3}\right)
 \leq 
 \min\left(N^3, \frac{1}{(1+|\zeta_-|)(1-|\zeta_-|)^3}\right) \\
 &\leq 
 \min\left(N^3, \frac{1}{(1-|\zeta_-|)^3}\right) 
 \leq 
  \min\left(N, \frac{1}{(1-|\zeta_-|)}\right)^3.
 \end{split}
\end{equation*}
Thus, 
\begin{equation*}
\begin{split}
 \|E_+\|^2=\|E_-\|^2 
& \leq \frac{1}{|a|^2}
 \min\left(
 \min\left(N, \frac{1}{(1-|\zeta_-|)}\right)^3, 
 \frac{4}{|1-\zeta_+/\zeta_-|^2}F_N(|\zeta_-|)
 \right)\\
 & \leq \frac{1}{|a|^2}
 \min\left(
 \min\left(N, \frac{1}{(1-|\zeta_-|)}\right)^3, 
 \frac{4}{|1-\zeta_+/\zeta_-|^2}
 \min\left( N,\frac{2}{1-|\zeta_-|}\right)
 \right)\\
 & = \frac{1}{|a|^2}\min\left( N,\frac{2}{1-|\zeta_-|}\right)
 \min\left(N, \frac{1}{(1-|\zeta_-|)}, \frac{2}{|1-\zeta_+/\zeta_-|}
 \right)^2, 
  \end{split}
\end{equation*}
and we conclude \eqref{estgrunp.7}.
\end{proof}

In the following we concentrate on the case when $E_1$ is a true 
non-degenerate ellipse, i.e. when 
\begin{equation}\label{egnp.1}
  0 < |b| < |a|.
\end{equation}
The degeneration $\zeta_+/\zeta_- \approx 1 $ takes place near the 
focal points $z=\pm2\sqrt{ab}$ where 
$\zeta_+ \approx \zeta_- \approx \pm \sqrt{b/a}$ and hence $|\zeta_-| < 1$ 
so we are away from the degeneration $|\zeta_-| \approx 1$, which takes 
place near $E_1$. Until further notice we assume that $z$ is not in a
neighbourhood of the focal points.
\par
From \eqref{estgrunp.6}, \eqref{estgrunp.7}, we get 
\begin{equation}\label{egnp.2}
  \| E\| \leq \mO(1) F_N(|\zeta_-|),
\end{equation}
\begin{equation}\label{egnp.3}
 \|E_+\|, ~\|E_-\|\leq \mO(1)F_N(|\zeta_-|)^{\frac{1}{2}}.
\end{equation}
Here, we used that $F_{k+1}(t) \asymp \min(k+1,\frac{1}{1-t})$ for 
$0\leq t\leq 1$. 
\subsection{When one of $|\zeta _\pm|$  is larger than
  1 and the other smaller than 1.} In this case we
estimate the resolvent of $P=P_\mathrm{I}$ directly. Recall that we work with 
$P=P_\mathrm{I}=1_{[1,N]}(a\tau ^{-1}+b\tau )$, $\tau =\tau _1$ with
the identification $\ell^2([1,N])\simeq \ell^2_{[1,N]}({\bf Z})$. We
start by deriving a fairly explicit expression for the resolvent,
valid under the sole assumption that $z$ is not in the spectrum.

\medskip\par\noindent a) We first invert $a\tau ^{-1}+b\tau -z$ on
$\ell^2({\bf Z})$. This is a convolution operator and we look for a
fundamental solution $F:{\bf Z}\to {\bf C}$ solving
\begin{equation}\label{1}
(a\tau^{-1}+b\tau -z)F=\delta _0,
\end{equation}
where $\delta _0(j)=\delta _{0,j}$. As before, we 
assume that $|\zeta _+|\le |\zeta _-|$. When $|\zeta _+|<1<|\zeta _-|$
our function $F$ will belong to $\ell^1$. Try
\begin{equation}\label{2}
F(k)=c\begin{cases}\zeta _+^k,\ k\ge 0,\\
\zeta _-^k,\ k\le 0,
\end{cases}
\end{equation}
where $c$ will be determined. (\ref{1}) means that
\begin{equation}\label{3}
au(k+1)+bu(k-1)-zu(k)=\delta _{0,k},\ k\in {\bf Z}.
\end{equation}
With the choice (\ref{2}), this holds for $k\ne 0$ and for $k=0$ we
get
$$
c(a\zeta _++b\zeta _-^{-1}-z)=1,
$$
i.e.
\begin{equation}\label{4}
c=\frac{1}{a\zeta _++b/\zeta _--z}.
\end{equation}
Using that $a\zeta _\pm + b/\zeta _\pm -z=0$, we have $b/\zeta
_--z=-a\zeta _-$ and (\ref{4}) becomes
\begin{equation}\label{5}
c=\frac{1}{a(\zeta _+-\zeta _-)},
\end{equation}
assuming of course that $\zeta _+\ne \zeta _-$ which follows
from the assumption that $z$ is not in the spectrum of
$P_\mathrm{I}$. 

\par Thus, with $P_\infty =a\tau ^{-1}+b\tau $ acting on functions on ${\bf
  Z}$, we get 
\begin{equation}\label{6}\begin{split}
(P_\infty -z)\circ (F*)v=v,\ \ v\in \ell^2_\mathrm{comp}({\bf Z}),\\
(F*)\circ (P_\infty -z) u=u,\ \ u\in \ell^2_\mathrm{comp}({\bf Z}),
\end{split}\end{equation}
where $\ell^2_\mathrm{comp}$ is the space of functions on ${\bf Z}$ that vanish
outside a bounded interval, and $F*$ denotes the convolution operator,
defined by $F*v(j)=\sum_kF(j-k)v(k)$. When 
$$|\zeta _+|<1<|\zeta _-|,$$
$F$ belongs to $\ell^1$, $F*$ is bounded on $\ell^2$, (\ref{6})
extends to the case when $u,v\in \ell^2({\bf Z})$ and then expresses that
$F*$ is a bounded 2-sided inverse of $P_\infty -z:\ell^2\to \ell^2$.

For future reference we combine (\ref{2}) and (\ref{5}) to
\begin{equation}\label{7}
F(k)=\frac{1}{a(\zeta _+-\zeta _-)}\begin{cases}\zeta _+^k,\ k\ge 0,\\
\zeta _-^k,\ k\le 0.
\end{cases}
\end{equation}

\medskip
\par\noindent b) We next solve 
$$
(a\tau ^{-1}+b\tau -z)u=0\hbox{ on }{\bf Z},
$$
with one of the two sets of ``Dirichlet'' conditions,
\begin{equation}\label{L}
u(0)=1,\ u(N+1)=0
\end{equation}
or
\begin{equation}\label{R}
u(0)=0,\ u(N+1)=1
\end{equation}
Denote the solutions by $u=u_L$, $u=u_R$ respectively, when they
exist and are unique.

\par In both cases we know that $u$ has to be of the form
$$u(j)=c_+\zeta _+^j+c_-\zeta _-^j,$$
and it suffices to see when $c_\pm $ exist and are unique. After some
straight forward calculations, we get existence and uniqueness under
the condition
\begin{equation}\label{8}
\zeta _+^{N+1}-\zeta _-^{N+1}\ne 0,
\end{equation}
and then
\begin{equation}\label{9}
u_L(j)=\frac{1}{1-(\zeta _+/\zeta _-)^{N+1}}\left(\zeta _+^j-\zeta
  _+^{N+1}(1/\zeta _-)^{N+1-j} \right),
\end{equation}
\begin{equation}\label{10}
u_R(j) =\frac{1}{1-(\zeta _+/\zeta _-)^{N+1}}\left(
(1/\zeta _-)^{N+1-j}-(1/\zeta _-)^{N+1}\zeta _+^j\right) .
\end{equation}

\medskip\par\noindent c) Solution of $(P-z)u=v$ in $\ell^2([1,N])$. We
adopt the assumption (\ref{8}) from now on and recall, that this is
equivalent to the assumption that $z$ avoids the spectrum of
$P=P_\mathrm{I}$ and the two focal points. With the usual identification $\ell^2([1,N])\simeq
\ell^2_{[1,N]}({\bf Z})$ it is now clear that the unique solution is 
$$
u=\widetilde{u}_{\vert_{[1,N]}},\hbox{ where }\widetilde{u}=F*v-(F*v)(0)u_L-(F*v)(N+1)u_R.
$$

\par Let $E=(P-z)^{-1}$ and let $E(j,k)$, $1\le j,k\le N$ be the
matrix elements of $E$. Then $E(j,k)=\widetilde{u}(j)$ where
$\widetilde{u}$ is the function above associated to $v=\delta _k$. 
Writing $\zeta _{\mathrm{sgn(j)}}^j
=\zeta _+^j$ for $j\ge 0$ and $=\zeta _-^j$ for $j<0$, $|\zeta
_{\mathrm{sgn}(j)}|^j=|\zeta _{\mathrm{sgn}(j)}^j|$, we get first
$$
E(j,k)=F(j-k)-F(-k)u_L(j)-F(N+1-k)u_R(j),
$$
and after substitution of the above expressions for $F$, $u_L$ and
$u_R$,
\begin{equation}\label{11}
\begin{split}
E(j,k)=&\frac{1}{a(\zeta _+-\zeta _-)}\times \\
&\Bigg(\zeta _{\mathrm{sgn}(j-k)}^{j-k}-\frac{1}{1-\left(\frac{\zeta _+}{\zeta
    _-}\right)^{N+1}}
\left(\left(1-\left(\frac{\zeta _+}{\zeta _-} \right)^{N+1-j}
  \right)\zeta _+^j\left(\frac{1}{\zeta _-} \right)^k\right.\\
& \left. -\left(1-\left(\frac{\zeta _+}{\zeta _-} \right)^j \right)
\left(\frac{1}{\zeta _-} \right)^{N+1-j}\zeta _+^{N+1-k}
 \right)
 \Bigg) .
\end{split}
\end{equation}

\medskip
\par\noindent d) In addition to (\ref{8}), we now assume
\begin{equation}\label{12}
|\zeta _+|\le 1\le |\zeta _-|.
\end{equation}
Then we get
\begin{equation}\label{13}
\begin{split}
&|E(j,k)|\le \frac{1}{|a||\zeta _+-\zeta _-|}\times \\
&\left( |\zeta _{\mathrm{sgn}(j-k)}|^{j-k}+\frac{2}{\left|
      1-\left(\frac{\zeta _+}{\zeta _-} \right)^{N+1} \right|}
\left(|\zeta _+|^j\left(\frac{1}{|\zeta _-|}
  \right)^k+\left(\frac{1}{|\zeta _-|} \right)^{N+1-j}|\zeta _+|^{N+1-k} \right)
 \right).
\end{split}
\end{equation}

\par In the big parenthesis the first term corresponds to a convolution and
the second term corresponds to the sum of two rank 1
operators. Letting $\|\cdot \|$ denote the norm in $\ell^2$ or in
${\mathcal{L}}(\ell^2,\ell^2 )$, depending on the context, we get
\begin{equation}\label{14}
\| E\|\le \frac{1}{|a||\zeta _+-\zeta _-|}\left(\sum_{1-N}^{N-1} |\zeta
  _{\mathrm{sgn}(j)}|^j
+\frac{4}{\left| 1-\left(\frac{\zeta _+}{\zeta _-} \right)^{N+1}
  \right|}
\|1_{[1,N]}\zeta _+^\cdot\|\|1_{[1,N]}\zeta _-^{-\cdot }\| 
 \right).
\end{equation}
Recall that $F_N(t)=1+t+...+t^{N-1}$ and that
\begin{equation}\label{15}
F_N(t)\asymp \min (1/(1-t),N), 0<t\le 1.
\end{equation} 
We have
\begin{equation}\label{16}
\sum_{1-N}^{N-1} |\zeta _{\mathrm{sgn}(j)}|^j=1+|\zeta _+|F_{N-1}(|\zeta
_+|)+\frac{1}{|\zeta _-|}F_{N-1}(1/|\zeta _-|).
\end{equation}

\par Also,
$$
\| 1_{[1,N]}\zeta _+^\cdot \|^2=|\zeta _+|^2F_N(|\zeta _+|^2).
$$
Here,
$$
F_N(t^2)=\frac{1-t^{2N}}{1-t^2}=\frac{1+t^N}{1+t}F_N(t),
$$
so 
$$
F_N(t)/2\le F_N(t^2)\le F_N(t).
$$
Similarly,
$$
\| 1_{[1,N]}\zeta _-^{-\cdot }\|^2=\frac{1}{|\zeta _-|^2}F_N(1/|\zeta _-|^2)
$$
and using these facts in (\ref{14}), we get
\begin{equation}\label{17}
\begin{split}
&\| E\|\le \frac{{\mO}(1)}{|a||\zeta _+-\zeta _-|}\times \\
&\left( 1+|\zeta _+|F_N(|\zeta _+|)+\frac{1}{|\zeta
    _-|}F_N(\frac{1}{|\zeta _-|} )+\frac{|\zeta _+/\zeta
  _-|}{|1-(\zeta _+/\zeta _-)^{N+1}|}F_N(|\zeta
_+|)^{\frac{1}{2}}F_N(\frac{1}{|\zeta _-|})^{\frac{1}{2}}\right) .
\end{split}
\end{equation}

\section{Grushin problem for the perturbed operator}\label{grpp}
We interested in the following random perturbation of $P_0=P_\mathrm{I}$: 
\begin{equation}\label{grpp.0}
 P_{\delta} := P_0 + \delta Q_{\omega}, 
 \quad
 Q_{\omega}=(q_{j,k}(\omega))_{1\leq j,k\leq N}, 
\end{equation}
where $0\leq\delta\ll 1 $ and $q_{j,k}(\omega)$ are independent and 
identically distributed complex Gaussian random variables, 
following the complex Gaussian law $\mathcal{N}_{\C}(0,1)$. 
\\
\par
The Markov inequality implies that if $C_1>0$ is large enough, then 
for the Hilbert-Schmidt norm, 
\begin{equation}\label{grpp.0b}
\mathds{P}\left[ 
\Vert Q_{\omega}\Vert_\mathrm{HS}\le C_1N
\right] \ge 1-e^{-N^2}.
\end{equation}
This has already been observed by W.~Bordeaux-Montrieux in \cite{BM}.
\subsection{A general discussion}\label{grpp0}
We begin with a formal 
discussion of the natural Grushin problem for $P_{\delta}$. 
Recall from Section \ref{grunp} that the 
Grushin problem is of the form 
\begin{equation*}
 \mathcal{P}_0=\begin{pmatrix}
                 R_+ & 0 \\
                 P_0-z & R_-\\
                \end{pmatrix} 
                :~\C^N\times\C \longrightarrow \C\times\C^N,
\end{equation*}
where we added a subscript $0$ to indicate that we deal with the 
unperturbed operator. Recall that $\mathcal{P}_0$ is bijective with inverse
\begin{equation*}
 \mathcal{E}_0=\begin{pmatrix}
                 E_+^0 & E^0 \\
                 E_{-+}^0 & E_-^0\\
                \end{pmatrix}
                :~\C\times\C^N \longrightarrow \C^N\times\C,
\end{equation*}
where we added a superscript $0$ for the same reason. 
If $\delta \|Q_{\omega}\| \|E^0\| <1$, we see using a 
Neumann series that 
\begin{equation*}
 \mathcal{P}_{\delta} =
	       \begin{pmatrix}
                 R_+ & 0 \\
                 P_{\delta}-z & R_-\\
                \end{pmatrix}
                :~\C^N\times\C \longrightarrow \C\times\C^N,
\end{equation*}
is bijective and admits the inverse
\begin{equation*}
 \mathcal{E}_{\delta}=\begin{pmatrix}
                 E_+^{\delta} & E^{\delta} \\
                 E_{-+}^{\delta} & E_-^{\delta}\\
                \end{pmatrix}
                :~\C\times\C^N \longrightarrow \C^N\times\C.
\end{equation*}
where 
\begin{equation}\label{grpp.1.1}
\begin{split}
 &E_+^{\delta} 
 = 
 E_+^0 - \delta E^0Q_{\omega}E_+^0 + \delta^2(E^0Q_{\omega})^2 E_+^0 + \dots 
 = (1 + \delta E^0Q_{\omega})^{-1}E_+^0, \\
 &E_-^{\delta} 
 = 
 E_-^0 - \delta E_-^0(Q_{\omega})E^0 + \delta^2 E_-^0(Q_{\omega}E^0)^2  + \dots 
 = E_-^0(1 + \delta Q_{\omega}E^0)^{-1},\\
 &E^{\delta} 
 = 
 E^0 - \delta E^0(Q_{\omega}E^0) + \delta^2E^0(Q_{\omega}E^0)^2 + \dots 
 = E^0(1 + \delta Q_{\omega}E^0)^{-1}, \\
 &E_{-+}^{\delta} 
 = 
 E_{-+}^0 - \delta E_-^0Q_{\omega}E_+^0 + \delta^2E_-^0Q_{\omega}E^0Q_{\omega} E_+^0 + \dots \\
 &\phantom{E_{-+}^{\delta}}= E_{-+}^0 - \delta E_-^0 Q_{\omega}(1 + \delta E^0Q_{\omega})^{-1}E_+^0. 
 \end{split}
\end{equation}
One obtains the following estimates
\begin{equation}\label{grpp.1.2}
\begin{split}
 & \|E^{\delta}\| \leq 
 \frac{\| E^0\|}{1 - \delta \|Q_{\omega}\| \|E^0\|}, ~
 \|E_{\pm}^{\delta}\| \leq 
 \frac{\| E_{\pm}^0\|}{1 - \delta \|Q_{\omega}\| \|E^0\|}, \\
 & | E_{-+}^{\delta} - E_{-+}^{0}| \leq  
 \frac{\delta\| E_{+}^0\| \| E_{-}^0\| \|Q_{\omega}\|}{1 - \delta \|Q_{\omega}\| \|E^0\|}.
\end{split}
\end{equation}
Differentiating the equation $\mathcal{E}^{\delta}\mathcal{P}^{\delta}=1$ with 
respect to $\delta$ yields
\begin{equation}\label{grpp.1.3}
 \partial_{\delta}\mathcal{E}^{\delta} = 
 - \mathcal{E}^{\delta}(\partial_{\delta}\mathcal{P}^{\delta})\mathcal{E}^{\delta}
 =
 -\begin{pmatrix}
   E^\delta Q_{\omega} E_+^{\delta} & E^\delta Q_{\omega} E^{\delta} \\
   E_-^\delta Q_{\omega} E_+^{\delta} & E_-^\delta Q_{\omega} E^{\delta} \\
  \end{pmatrix}.
\end{equation}
Integrating this relation from $0$ to $\delta$ yields 
\begin{equation}\label{grpp.1.4}
 \|E^{\delta} -E^0\| \leq 
 \frac{\delta\|Q_{\omega}\| \| E^0\|^2}{(1 - \delta \|Q_{\omega}\| \|E^0\|)^2}, ~
 \|E_{\pm}^{\delta} -E_{\pm}^0\| \leq 
 \frac{\delta\|Q_{\omega}\| \| E_{\pm}^0\|\|E^0\|}{(1 - \delta \|Q_{\omega}\| \|E^0\|)^2}.
\end{equation}
\par
Since $\mathcal{P}^{\delta}$ is invertible and of finite 
rank, we know that 
\begin{equation*}
 |\partial_{\delta} \ln\det\mathcal{P}^{\delta}| 
 = |\mathrm{tr}(\mathcal{E}^{\delta}\partial_{\delta}\mathcal{P}^{\delta})|.
\end{equation*}
Letting $\|\cdot\|_{\tr}$ denote the trace class norm, we get 
\begin{equation}\label{grpp.1.5}
 |\partial_{\delta} \ln\det\mathcal{P}^{\delta}| 
 = |\mathrm{tr}(Q_{\omega}E^{\delta})|
 \leq \|Q_{\omega}\|_{\mathrm{tr}} \|E^{\delta}\|
 \leq \frac{\| E^0\| \|Q_{\omega}\|_{\mathrm{tr}}}{1 - \delta \|Q_{\omega}\| \|E^0\|},
\end{equation}
where $\|Q_{\omega}\|_{\mathrm{tr}} \leq N^{1/2}\|Q_{\omega}\|_{\HS}$. Integration from 
$0$ to $\delta$ yields
\begin{equation}\label{grpp.1.6}
 \left| \ln |\det\mathcal{E}^{\delta}| - \ln |\det\mathcal{E}^{0}| \right| = 
 \left| \ln |\det\mathcal{P}^{\delta}| - \ln |\det\mathcal{P}^{0}| \right|
 \leq 
 \frac{\delta \| E^0\| \|Q_{\omega}\|_{\mathrm{tr}}}{1 - \delta \|Q_{\omega}\| \|E^0\|}.
\end{equation}
Sharpening the assumption $\delta \|Q_{\omega}\| \|E^0\| < 1$ to 
\begin{equation}\label{grpp.1.7}
  \delta \|Q_{\omega}\| \|E^0\| < \frac{1}{2},
\end{equation}
we get
\begin{equation}\label{grpp.1.8}
 \|E^{\delta}\| \leq 2\| E^0\|, ~
 \|E_{\pm}^{\delta}\| \leq 2\| E_{\pm}^0\|,~ 
 | E_{-+}^{\delta} - E_{-+}^{0}| \leq  
 2\delta \| E_{+}^0\| \| E_{-}^0\| \|Q_{\omega}\|.
\end{equation}
By \eqref{grpp.1.3} we know that 
$\partial_{\delta}E_{-+}^{\delta} = - E_-^{\delta} Q_{\omega} E_+^{\delta}$. 
Therefore, using \eqref{grpp.1.2}, \eqref{grpp.1.4} and \eqref{grpp.1.8} we get 
\begin{equation}\label{grpp.1.9}
\begin{split}
 |\partial_{\delta}E_{-+}^{\delta} +E_-^0 Q_{\omega} E_+^0| 
  &\leq 
  \|E_-^0 Q_{\omega} \| \|E_+^{\delta}-E_+^0\| + 
  \|Q_{\omega} E_+^{\delta}  \| \|E_-^{\delta}-E_-^0\| \\
  &\leq 12 \delta \|Q_{\omega}\|^2 \|E_-^0\| \|E_+^0\| \| E^0\|.
  \end{split}
\end{equation}
By integration from $0$ to $\delta$, we conclude 
\begin{equation}\label{grpp.1.10}
 E_{-+}^{\delta} = E_{-+}^{0} - \delta E_-^0 Q_{\omega} E_+^0 
 + \mO(1)\delta^2 \|Q_{\omega}\|^2 \|E_-^0\| \|E_+^0\| \| E^0\|.
\end{equation}
\subsection{More specific estimates}\label{grpp.I}
$\phantom{.}$\\
\textbf{a) The case where $z$ is inside the ellipse $E_1$.} We adopt the 
non-degeneracy condition \eqref{egnp.1}: $0<|b|<|a|$ and keep the
assumption that $z$ avoids a neighbourhood of the focal points. 
In view of \eqref{egnp.2} and the fact that 
$\|Q_{\omega}\|_{\HS}\leq C_1 N$ (cf. \eqref{grpp.0b}) we replace 
assumption \eqref{grpp.1.7} by the stronger and more explicit 
condition  
\begin{equation}\label{grpp.1}
 \delta N F_N(|\zeta_-|) \ll 1.
\end{equation}
Notice that this is fulfilled for all $z$ inside $E_1$, if we make 
the even stronger assumption 
\begin{equation}\label{grpp.2}
 \delta N^2 \ll 1.
\end{equation}
(Recall that $N\gg 1$). 
\par 
We conclude from the discussion above and \eqref{egnp.2}, \eqref{egnp.3}:
\begin{prop}\label{grpp1}
 Let $0\leq \delta \ll 1$ satisfy \eqref{grpp.1} and let 
 $P_{\delta}$ be as in \eqref{grpp.0}, $R_{\pm}$ be as 
 in \eqref{grunp.1} and assume that $\|Q_{\omega}\|_{\HS}\leq C_1 N$ 
 (cf. \eqref{grpp.0b}). Then, 
 \begin{equation*}
 \mathcal{P}_{\delta} =
	       \begin{pmatrix}
                 R_+ & 0 \\
                 P_{\delta}-z & R_-\\
                \end{pmatrix}
                :~\C^N\times\C \longrightarrow \C\times\C^N,
\end{equation*}
is bijective with bounded inverse
\begin{equation*}
 \mathcal{E}_{\delta}=\begin{pmatrix}
                 E_+^{\delta} & E^{\delta} \\
                 E_{-+}^{\delta} & E_-^{\delta}\\
                \end{pmatrix}
                :~\C\times\C^N \longrightarrow \C^N\times\C,
\end{equation*}
where 
\begin{equation}\label{grpp.3}
 \| E^{\delta} -E^0\| \leq \mO(1)\delta N F_{N}(|\zeta_-|)^2,
\end{equation}
\begin{equation}\label{grpp.4}
 \| E^{\delta}_{\pm} -E^0_{\pm}\| \leq \mO(1)\delta N F_{N}(|\zeta_-|)^{3/2},
\end{equation}
\begin{equation}\label{grpp.5}
 E_{-+}^{\delta} 
 = 
 E_{-+}^0 - \delta E_-^0Q_{\omega}E_+^0
 +\mO(1)\left(\delta N F_{N}(|\zeta_-|)\right)^{2}.
\end{equation}
Here $E_{\pm}^0$, $E^0$, $E_{-+}^0$ are as in \eqref{estgrunp.1}, 
\eqref{estgrunp.2}, \eqref{grunp.10}.
\end{prop}
Using \eqref{estgrunp.2}, \eqref{estgrunp.3} we get with $Q_{\omega}=(q_{j,k}(\omega))$, 
\begin{equation}\label{grpp.6}
\begin{split}
 E_{-+}^{\delta} =  \frac{\zeta _+^{N+1}-\zeta _-^{N+1}}{a(\zeta _+-\zeta _-)}
 &-\delta \sum_{j,k=1}^N q_{j,k}(\omega)
 \frac{\zeta_+^{N+1-j}-\zeta_-^{N+1-j}}{a(\zeta_+-\zeta_-)}
 \frac{\zeta_+^{k}-\zeta_-^{k}}{a(\zeta_+-\zeta_-)}\\
 &+\mO(1)\left(\delta NF_{N}(|\zeta_-|)\right)^{2}.
\end{split}
\end{equation}
From \eqref{grpp.1.6}, we get 
\begin{prop}
 Under the same assumptions as in Proposition \ref{grpp1}, we have 
 \begin{equation}\label{grpp.7}
  \left| \ln |\det \mathcal{P}_{\delta}| - \ln|\det\mathcal{P}_0|\right| 
  \leq 
  \mO(1)\delta N^{3/2} F_{N}(|\zeta_-|).
 \end{equation}
\end{prop}
Here we also used that 
$\|Q_{\omega}\|_{\tr} \leq N^{1/2}\|Q_{\omega}\|_{\HS} \leq \mO(N^{3/2})$. 
Also recall that $\det \mathcal{P}_0 = a^{N+1}$. 
\\
\\
\textbf{b) The case when $z$ belongs to a compact set outside $E_1$.} We 
have seen in \eqref{17} that 
\begin{equation}\label{grpp.8}
 \|(P-z)^{-1}\| \leq \mO(1)F_N\left(\frac{1}{|\zeta_-|}\right).
\end{equation}
Assume 
\begin{equation}\label{grpp.9}
 \delta N F_N\left(\frac{1}{|\zeta_-|}\right) \ll 1,
\end{equation}
which like (\ref{grpp.1}) is a weaker condition than (\ref{grpp.2}).
Then, $\|\delta Q_{\omega} (P-z)^{-1}\| \leq \mO(1)\delta N F_N \ll 1$ 
and $P_{\delta}-z$ is bijective satisfying
\begin{equation}\label{grpp.10}
 \|(P_{\delta}-z)^{-1}\| \leq \mO(1) F_N\left(\frac{1}{|\zeta_-|}\right),
\end{equation}
\begin{equation}\label{grpp.11}
 \|(P_{\delta}-z)^{-1}-(P-z)^{-1}\| \leq \mO(1) \delta N (F_N(|\zeta_-|^{-1}))^2,
\end{equation}
In analogy with \eqref{grpp.1.6}, we have 
\begin{equation*}
 \left| \ln |\det(P_{\delta}-z)| - \ln|\det(P-z)|\right| 
 \leq 
 \mO(1)\delta \|Q_{\omega}\|_{\tr} F_N\left(\frac{1}{|\zeta_-|}\right),
\end{equation*}
leading to 
\begin{equation}\label{grpp.12}
 \left| \ln |\det(P_{\delta}-z)| - \ln|\det(P-z)|\right| 
 \leq 
 \mO(\delta)N^{3/2}F_N\left(\frac{1}{|\zeta_-|}\right),
\end{equation}
under the assumption $\|Q_{\omega}\|_{\HS}\leq \mO(N)$. 
Recall that $\det(P_0 -z)$ is given by \eqref{grunp.13}. 
\\ \\
\textbf{c) Estimation of the probability that $E_{-+}^{\delta}$ is small.} 
We now return to the situation in a), i.e. when $z$ is inside the ellipse $E_1$, 
so that $|\zeta_-|\leq 1$. We assume \eqref{grpp.1} (to be strengthened later on). 
We shall follow Section 13.5 in \cite{Sj15} with only small changes. 
Write \eqref{grpp.6} as 
\begin{equation}\label{grpp.13} 
 E_{-+}^{\delta} = 
 \frac{\zeta_+^{N+1} - \zeta_-^{N+1}}{a(\zeta_+ - \zeta_-)} 
 - \delta (Q_{\omega}| \overline{Z}) + \mO(1)(\delta N F_N)^2,
\end{equation}
$F_N=F_N(|\zeta_-|)$, where 
\begin{equation}\label{grpp.14}
 Z = \left( 
 \frac{\zeta_+^{N+1-j} - \zeta_-^{N+1-j}}{a(\zeta_+ - \zeta_-)} 
 \frac{\zeta_+^{k} - \zeta_-^{k}}{a(\zeta_+ - \zeta_-)} 
 \right)_{1\leq j,k\leq N}.
\end{equation}
In the following we often write $| \cdot|$ for the Hilbert-Schmidt norm 
(i.e. the $l^2$-norm of the matrix). Write 
\begin{equation*}
 Z=\left( 
 F_{N+1-j}(\zeta_+/\zeta_-)F_{k}(\zeta_+/\zeta_-) \zeta_-^{N-j+k-1}
 \right)_{1\leq j,k\leq N}
\end{equation*}
and assume 
\begin{equation}\label{grpp.15}
 z\notin \text{neigh}([-2\sqrt{ab},2\sqrt{ab}],\C)
\end{equation}
so that 
\begin{equation}\label{grpp.16}
 |\zeta_+/\zeta_-| \leq 1 - \frac{1}{\mO(1)}.
\end{equation}
(In fact, using that $\zeta _-=b/(a\zeta _+)$, the assumption $|\zeta
_-|=|\zeta _+|$ leads to $\zeta _+=\sqrt{b/a}e^{i\theta }$, for some
$\theta\in {\bf R}$, so $z=a\zeta _++b/\zeta _+=2\sqrt{ab}\cos \theta
$ belongs to the focal segment.)
Then $|F_{N+1-j}(\zeta_+/\zeta_-)|, |F_k(\zeta_+/\zeta_-)| \asymp 1$ and 
a straight forward calculation shows that 
\begin{equation}\label{grpp.17}
 \frac{1}{\mO(1)} F_N(|\zeta_-|) \leq |Z| \leq \mO(1)F_N(|\zeta_-|).
\end{equation}
Working still under the assumption that $|Q_{\omega}| \leq \mO(N)$, we 
get (cf. \eqref{grpp.1.8})
\begin{equation}\label{grpp.18}
 |E_{-+}^{\delta} - E_{-+}^0 | \leq \mO(1)\delta N F_N(|\zeta_-|).
\end{equation}
From \eqref{grpp.13} and the Cauchy inequalities, we get 
\begin{equation}\label{grpp.19}
 d_Q E_{-+}^{\delta} = \delta |Z| (dQ|e_1) + \mO(N^{-1})(\delta N F_{N})^2
\end{equation}
in $\C^{N^2}$, where 
\begin{equation}\label{grpp.20}
 e_1 = \frac{1}{|Z|}\overline{Z}.
\end{equation}
Complete $e_1$ into an orthonormal basis $e_1,e_2,\dots,e_{N^2}$ in $\C^{N^2}$ 
and write 
\begin{equation*}
 Q = Q' + Q_1 e_1, \quad 
 Q' = \sum_2^{N^2} Q_k e_k \in (e_1)^{\perp}.
\end{equation*}
Then \eqref{grpp.13}, \eqref{grpp.19} read 
\begin{equation}\label{grpp.21}
 E_{-+}^{\delta} = E_{-+}(0) + \delta |Z| Q_1 + \mO(1)(\delta N F_N)^2,
\end{equation}
\begin{equation}\label{grpp.22}
 d_Q E_{-+}^{\delta} = \delta |Z| dQ_1 + \mO(N^{-1})(\delta N F_N)^2.
\end{equation}
\par
As in \cite[Chapter  13]{Sj15} we can extend $Q\mapsto  E_{-+}^{\delta}(Q)$ 
to a smooth function $F:\C^{N^2}\to \C$ such that 
\begin{equation}\label{grpp.23}
 F(Q) = E_{-+}(0) + \delta |Z|Q_1 + \mO(1)(\delta N F_N)^2 
      =: F(0) +\delta|Z|f(Q)
\end{equation}
\begin{equation}\label{grpp.24}
 d_Q F(Q) = \delta |Z| dQ_1 + \mO(N^{-1})(\delta N F_N)^2 
\end{equation}
and such that the remainders vanish outside $B_{\C^{N^2}}(0,2C_1N)$, where 
$B_{\C^{N^2}}(0,C_1N)$ is the ball of validity for \eqref{grpp.21}, \eqref{grpp.22}. 
The function $f$ satisfies
\begin{equation}\label{grpp.25}
 f(Q)= Q_1 + \mO(\delta N^2 F_N),
\end{equation}
\begin{equation}\label{grpp.26}
 d_Qf= dQ_1 + \mO(\delta N F_N).
\end{equation}
From the assumption \eqref{grpp.1} it follows that the map 
$\C \ni Q_1 \mapsto f(Q_1,Q')\in\C$ is bijective for every 
$Q'$ and has a smooth inverse $g=g(\zeta,Q')$, satisfying 
\begin{equation}\label{grpp.27}
 g(\zeta,Q') = \zeta + \mO(\delta N^2  F_N),
\end{equation}
\begin{equation}\label{grpp.28}
 d_{\zeta,Q'} g(\zeta,Q') = d\zeta + \mO(\delta N  F_N).
\end{equation}
\par
Let $\mu(d\zeta)$ be the direct image under $f$ of the Gaussian 
measure $\pi^{-N^2}\e^{-|Q|^2}L(dQ)$. We study $\mu$ in $D(0,C)$ 
for any fixed $C>0$. For $\varphi\in\mathcal{C}_0(D(0,C))$, we get 
\begin{equation*}
 \begin{split}
  &\int\varphi(\zeta)\mu(d\zeta) = 
  \int \varphi(f(Q))\pi^{-N^2}\e^{-|Q|^2}L(dQ) \\
  &=\int_{\C^{N^2-1}}\pi^{1-N^2}\e^{-|Q'|^2}
    \left(\int_{\C} \pi^{-1}\e^{-|Q_1|^2}\varphi(f(Q))L(dQ_1)\right)L(dQ')
    \\
  &=\int_{\C^{N^2-1}}\pi^{1-N^2}\e^{-|Q'|^2}
    \left[
    \int_{\C} \pi^{-1}\e^{-|g(\zeta,Q')|^2}\varphi(\zeta) L(d_\zeta g)
    \right]
    L(dQ'), 
  \end{split}
\end{equation*}
where
$$L(d_\zeta g)=L(dQ_1)=\det\left(\frac{\partial(Q_1,\overline{Q_1})}{\partial(\zeta,\overline{\zeta})}\right)
    L(d\zeta ).$$
We get for $\varphi \in \mathcal{C}_0(D(0,C))$, 
\begin{equation*}
 \begin{split}
 & \int \varphi(\zeta)\mu(d\zeta) \\
 & = \int_{\C}\varphi(\zeta) 
     \left(
     \int_{\C^{N^2-1}}
     \pi^{-1}\e^{-| g(\zeta ,Q')|^2}
     \pi^{1-N^2}\e^{-|Q'|^2}L(dQ')
    \right)L(d_\zeta g).
 \end{split}
\end{equation*}
so that in $D(0,C)$ 
\begin{equation}\label{grpp.29}
 \mu(d\zeta) = 
 \left( \int_{\C^{N^2-1}}(
     \pi^{-1}\e^{-| g(\zeta ,Q')|^2}
     \pi^{1-N^2}\e^{-|Q'|^2}L(dQ')\right)
     L(d_\zeta g)
\end{equation}
We conclude that for $|\zeta_0|, r\leq \mO(1)$, the probability that 
$|Q|\leq C_1 N$ and $f(Q) \in D(\zeta_0,r)$ is bounded from above by
$1-e^{-N^2}$ plus
\begin{equation}\label{grpp.30}
 \int_{{\bf C}^{N^2-1}}\int_{\zeta \in D(\zeta _0,r)}\pi
 ^{-1}e^{-|g(\zeta ,Q')|^2}L(d_\zeta  g)\pi ^{1-N^2}e^{-|Q'|^2}L(dQ').
\end{equation}
From (\ref{grpp.28}) we infer that 
$$
\{ g(\zeta ,Q');\, \zeta \in D(\zeta _0,r) \}\subset D(g(\zeta
_0,Q'),\widetilde{r}),\ \ \widetilde{r}=(1+{\mO}(\delta NF_N))r
$$
and the last integral is $\le$
$$
\int_{{\bf C}^{N^2-1}}\int_ {D(g(\zeta _0,Q'),\widetilde{r})} \pi ^1
e^{-|\omega |^2}L(d\omega )\pi ^{1-N^2}e^{-|Q'|^2}L(dQ').
$$
Here the inner integral is 
\begin{equation}\label{grpp.31}
 \leq \int_{D(0,\widetilde{r})} \frac{1}{\pi} \e^{-|\omega|^2} L(d\omega) 
 = 1- \e^{-\widetilde{r}^2}. 
\end{equation}
In fact, by rotation symmetry, we may assume that $\zeta_0=t \geq 0$ and 
by Fubini's theorem, we are reduced to show that $F(t) \leq F(0)$, where 
\begin{equation*}
 F(t) = \int_{t-\widetilde{r}}^{t+\widetilde{r}} \e^{-s^2} ds. 
\end{equation*}
It then suffices to observe that $F'(t)\leq 0$. 
\par
Thus the integral in (\ref{grpp.30}) is bounded by $(1+{\mO}
(F_N\delta N))(1-\e^{-\widetilde{r}^2})$. In terms of $E_{-+}^{\delta}$, 
we get under the assumption \eqref{grpp.1}: 
\begin{lemma}\label{grpp3}
 We recall \eqref{grpp.23}. For $0\leq t$, $|E_{-+}^0|\leq C\delta F_N(|\zeta_-|)$, 
 the probability that $|Q|\leq C_1 N$ and $|E_{-+}^{\delta}| \leq t$ is 
 \begin{equation*}
  \leq e^{-N^2} + (1 + \mO(F_N N \delta))\left(1 - \exp\left[-\left(\frac{t}{\delta|Z|}\right)^2\right]\right).
 \end{equation*}
\end{lemma}
From the bound 
\begin{equation}\label{grpp.32}
 |Q|\leq C_1 N, 
\end{equation}
that we adopt from now on, and the Cauchy-Schwartz inequality for the singular values of $Q$, we 
know that
\begin{equation}\label{grpp.33}
 \|Q\|_{\tr} \leq C_1 N^{3/2}.
\end{equation}
\section{Counting eigenvalues}\label{evI}
\subsection{Estimates on $\det(P_{\delta}-z)$ inside $E_1$}
In this section we assume (\ref{grpp.2}), implying (\ref{grpp.1}) when
$|\zeta _-|\le 1$.
We identify the eigenvalues of $P_{\delta}$ with the zeros of 
the function
\begin{equation}\label{evI.0}
 D_{\delta}(z)=\det(P_{\delta} -z).
\end{equation}
We sum up the various estimates and identities for this function: 
\par
When $\delta=0$, we have \eqref{grunp.13}
\begin{equation}\label{evI.1}
 \det(P_{0}-z ) = (-a)^{N} \frac{\zeta_+^{N+1}-\zeta_-^{N+1}}{\zeta_+ - \zeta_-} 
                = (-a)^{N} \zeta_-^{N} F_{N+1}(\zeta_+/\zeta_-),
\end{equation}
and for the Grushin problem \eqref{grunp.2} we have \eqref{grunp.12}: 
\begin{equation}\label{evI.2}
 \det\mathcal{P}_0(z) = a^{N+1} 
\end{equation}
and \eqref{grunp.10}
\begin{equation}\label{evI.3}
 E_{-+}^{0}(z) = \frac{\zeta_+^{N+1}-\zeta_-^{N+1}}{a(\zeta_+ - \zeta_-)} 
	       = \frac{\zeta_-^N}{a}F_{N+1}(\zeta_+/\zeta_-).
\end{equation}
For $z$ in the interior of $E_1$, by \eqref{grpp.1}, 
$\delta N F_N(|\zeta_-|)\ll 1$ and the assumption that $\|Q\|_{\HS}\leq C_1N$ 
we have \eqref{grpp.7}: 
\begin{equation}\label{evI.4}
 \ln|\det \mathcal{P}_{\delta}| = \ln|\det\mathcal{P}_0| + \mO(1)\delta N^{3/2} F_N(|\zeta_-|). 
\end{equation}
We also have the general identity (cf. \eqref{grunp.11}) 
\begin{equation}\label{evI.5}
 \det(P_{\delta}-z) = (-1)^N E_{-+}^{\delta}(z)\det \mathcal{P}_{\delta}(z).
\end{equation}
From \eqref{grpp.5}, \eqref{grpp.18} and \eqref{evI.3}, we infer that 
\begin{equation}\label{evI.6}
 |E_{-+}^{\delta}| \leq \frac{|\zeta_-|^N}{a}| F_{N+1}(\zeta_+/\zeta_-)| +
 \mO (1) \delta N F_{N}(|\zeta_-|).
\end{equation}
We will also assume that $z\notin \text{neigh}([-2\sqrt{ab},2\sqrt{ab}],\C)$ so that 
$|\zeta_+|\leq (1 - 1/\mO(1))|\zeta_-|$. Then \eqref{evI.6} 
implies that 
\begin{equation}\label{evI.7}
 |E_{-+}^{\delta}| \leq \mO(1), 
\end{equation}
and \eqref{evI.2},\eqref{evI.4}, \eqref{evI.5} give 
\begin{equation}\label{evI.8}
 \ln|\det(P_{\delta} -z)| \leq (N+1)\ln |a| + \mO(1)(1 + \delta N^{3/2} F_N(|\zeta_-|)).
\end{equation}
\par
Still under the assumption that $z$ is in the interior of $E_1$, we give a probabilistic lower 
bound on $\ln|\det(P_{\delta} -z)|$, starting from
\begin{equation}\label{evI.9}
 \begin{split}
  \ln|\det(P_{\delta} - z)| &= \ln |\det \mathcal{P}_{\delta}| 
    +\ln|E_{-+}^{\delta}| \\
  & \geq \ln| \det\mathcal{P}_0| + \ln|E_{-+}^{\delta}| 
    - \mO(1)\delta N^{3/2} F_N(|\zeta_-|) \\
  & = (N+1)\ln|a| + \ln|E_{-+}^{\delta}| - \mO(1)\delta N^{3/2} F_N(|\zeta_-|).
 \end{split}
\end{equation}
In order to apply Lemma \ref{grpp3}, we analyze the condition
\begin{equation}\label{evI.10}
 |E_{-+}^0(z)| \leq C\delta F_N(|\zeta_-|),
\end{equation}
which by \eqref{evI.3} amounts to 
\begin{equation*}
 |\zeta_-|^N| F_{N+1}(\zeta_+/\zeta_-)| \leq C\delta F_{N}(|\zeta_-|).
\end{equation*}
Since $F_{N+1}(\zeta_+/\zeta_-) = \mO(1)$ (by the assumption that $z$ avoids 
a neighborhood of the focal segment), this would follow from 
\begin{equation}\label{evI.11}
 |\zeta_-|^N \leq C\delta F_{N}(|\zeta_-|).
\end{equation}
Recall that
\begin{equation*}
 F_{N}(|\zeta_-|) \asymp \min\left( N, \frac{1}{1-|\zeta_-|}\right).
\end{equation*}
We know by (\ref{grpp.2}) that $\delta  F_{N}(|\zeta_-|) \ll N^{-1}$ so 
we are in the region where 
\begin{equation*}
 |\zeta_-|^N \ll \frac{1}{N},
\end{equation*}
i.e.
\begin{equation*}
 \ln|\zeta_-| \leq \frac{\ln(N^{-1}) - (\gg 1)}{N},
\end{equation*}
where $(\gg 1)$ indicates a sufficiently large constant 
and hence 
\begin{equation*}
 1 - |\zeta_-| \geq \frac{\ln N + (\gg 1)}{N}.
\end{equation*}
In this region $F_{N}(|\zeta_-|) =\frac{1}{1-|\zeta_-|}$, and to understand 
\eqref{evI.11} amounts to understanding for which $s$ ($=|\zeta_-|$) in 
$]\frac{1}{\mO(1)}, 1- \frac{\ln N + (\gg 1)}{N} ]$ we have 
\begin{equation}\label{evI.12}
 m(s)\leq C\delta,
\end{equation}
where 
\begin{equation}\label{evI.13}
 m(s) = s^N(1 -s ), \quad 0\leq s \leq 1. 
\end{equation}
This function increases from $s=0$ to
$s=s_{\max} = \frac{N}{N+1}= 1 - \frac{1}{N+1}$, and then 
decreases. We have
\begin{equation}\label{evI.14}
 m(s_{\max}) = \left(1 - \frac{1}{N+1}\right)^N \frac{1}{N+1} 
	     = \frac{1+\mO\!\left(\frac{1}{N}\right)}{\e N}
\end{equation}
while $\delta \ll \frac{1}{N}$. The solution $s=s_{\delta} < s_{\max}$ of
\begin{equation}\label{evI.14b}
 m(s) =C\delta
\end{equation}
satisfies 
\begin{equation*}
 s^N \frac{1}{N} \leq s^N (1 -s) = C\delta \leq s^N, 
\end{equation*}
so 
\begin{equation}\label{evI.15}
 (C\delta)^{\frac{1}{N}} \leq s_{\delta} \leq (C N \delta)^{\frac{1}{N}}.
\end{equation}
We now apply Lemma \ref{grpp3} to \eqref{evI.9} and get 
\begin{prop}\label{evI1}
 Restrict $z$ to the regions inside $E_1$, where 
 $\delta N F_N(|\zeta_-|) \ll 1$, $|\zeta_-|\leq s_{\delta}$ as in 
 \eqref{evI.14b}, \eqref{evI.15}. Then for each such $z$ and for 
 $t\leq C\delta F_N(|\zeta_-|)$, we have 
 \begin{equation}\label{evI.16}
  \ln|\det(P_{\delta}-z)| \geq (N+1)\ln|a| + \ln t - 
   \mO(1)\delta N^{3/2} F_N(|\zeta_-|)
 \end{equation}
 with probability
 \begin{equation}\label{evI.17}
  \begin{split}
  &\geq 1 - (1 + \mO(\delta N F_N))
  \left(1 -\exp\left(-\left[\frac{t}{\delta |Z|}\right]^2\right)\right)
  -\e^{-N^2}  \\
  & \geq 1 - \frac{t^2}{\mO(1)\delta^2 F_N(|\zeta_-|)^2} -\e^{-N^2}
  \end{split}
 \end{equation}
\end{prop}
\subsection{Estimates for $\det(P_{\delta} -z)$ in the exterior of $E_1$.} 
We just recall \eqref{grpp.12}: With probability $\geq 1 - \e^{-N^{2}}$, we 
have 
\begin{equation}\label{evI.18}
\ln |\det(P_{\delta}-z)| = N \ln|a| 
      + \ln \frac{|\zeta_+^{N+1} - \zeta_-^{N+1}|}{|\zeta_+- \zeta_-|}
     +\mO(\delta) N^{3/2}F_N(|\zeta_-|^{-1})
\end{equation}
for all $z$ satisfying \eqref{grpp.9}: 
\begin{equation}\label{evI.19}
 \delta N F_N \left(\frac{1}{|\zeta_-|}\right) \ll 1, 
\end{equation}
which is guaranteed for all $z$ in the exterior region by (\ref{grpp.2}).
Here we also used the formula \eqref{grunp.13}: 
\begin{equation*}
 \det (P-z) = (-a)^N\frac{\zeta_+^{N+1} - \zeta_-^{N+1}}{\zeta_+- \zeta_-}.
\end{equation*}
If $\delta N^2 \ll 1 $, then \eqref{evI.19} is satisfied in the whole exterior 
region. For larger values of $\delta$, \eqref{evI.19} says that 
\begin{equation*}
 \frac{\delta N}{1 - \frac{1}{|\zeta_-|}} \ll 1 ,
\end{equation*}
which means that $|\zeta_-| -1 \gg \delta N$. 
\subsection{Choice of parameters} 
We chose
\begin{equation}\label{evI.20}
 \delta \asymp N^{-\kappa}, \quad \kappa > \frac{5}{2},
\end{equation}
as in Theorem \ref{evI2}. Then \eqref{grpp.2} holds and by \eqref{evI.8} 
we have with probability $\geq 1 - \e^{-N^{2}}$ the upper bound 
\begin{equation}\label{evI.21}
 \ln|\det(P_{\delta} -z)| \leq N (\ln|a| + \mO(N^{-1})),
\end{equation}
for all $z$ in the interior of $E_1$, away from any fixed
neighborhood of the focal segment.
\par
Proposition \ref{evI1} is applicable for $|\zeta_-|\leq (C\delta)^{\frac{1}{N}}$ 
and hence for each $\zeta_-$ with 
\begin{equation}\label{evI.22}
 |\zeta_-| \leq \left(\frac{1}{\mO(1)}\right)^{\frac{1}{N}} N^{-\frac{\kappa}{N}} 
  = \e^{-\frac{\kappa}{N}(\ln N +\mO(1))}
\end{equation}
or equivalently 
\begin{equation}\label{evI.23}
 |\zeta_-| \leq 1 - \frac{\kappa}{N}(\ln N +\mO(1)),
\end{equation}
we have \eqref{evI.16}: 
\begin{equation}\label{evI.24}
 \ln|\det(P_{\delta}-z)| \geq N\ln|a| - \mO(1) +\ln t
\end{equation}
with probability 
\begin{equation}\label{evI.25}
 \geq 1 - \mO(1)N^{2\kappa} t^2 - \e^{-N^2}
\end{equation}
for $t\leq C^{-1}N^{-\kappa}$. Choose $t=\e^{-N^{\delta_0}}$ for a small 
fixed $\delta_0 >0$. Then, for each $\zeta_-$ satisfying \eqref{evI.23}, 
we have 
\begin{equation}\label{evI.26}
 \ln|\det(P_{\delta}-z)| \geq N(\ln|a| - 2N^{\delta_0 -1}),
\end{equation}
with probability 
\begin{equation}\label{evI.27}
 \geq 1 - \mO(1)N^{2\kappa} \e^{-2N^{\delta_0}}.
\end{equation}
\par
As for the exterior region we write 
\begin{equation*}
\begin{split}
 \ln \frac{|\zeta_+^{N+1}-\zeta_-^{N+1}|}{|\zeta_+ - \zeta_-|} &  = 
 N\ln|\zeta_-| + \ln\frac{|1 - (\zeta_+/\zeta_-)^{N+1}|}{|1 - \zeta_+/\zeta_-|} \\ 
 & = N\ln|\zeta_-| + \mO(1).
\end{split}
\end{equation*}
and from \eqref{evI.18},\eqref{evI.19} we get with probability $\geq 1 - \e^{-N^2}$, 
\begin{equation}\label{evI.28}
 \ln|\det(P_{\delta}-z)| = N(\ln|a| + \ln|\zeta_-| +\mO(N^{-1}))
\end{equation}
for all $z$ in any fixed bounded region with $|\zeta_-|\geq 1$. Put 
\begin{equation}\label{evI.29}
 \varphi(z) = \ln|a| +\max(\ln |\zeta_-|,0) + \frac{C}{N},
\end{equation}
for $C>0$ large enough. Then with probability $\geq 1 - \e^{-N^2}$ 
we have 
\begin{equation}\label{evI.30}
 \ln|\det(P_{\delta} - z)| \leq N\varphi(z),
\end{equation}
for all $z$ in any fixed compact subset of $\C$ which does not intersect 
the focal segment. 
\par
Moreover, 
\begin{equation}\label{evI.31}
 \ln|\det(P_{\delta}-z)| \geq N(\varphi(z) - \mO(N^{-1}))
\end{equation}
in the exterior region. For each $z$ with 
$|\zeta_-|\leq 1 -\frac{\kappa}{N}(\ln N +\mO(1))$, we have \eqref{evI.26} with 
probability as in \eqref{evI.27}: 
\begin{equation}\label{evI.32}
 \ln|\det(P_{\delta}-z)| \geq N(\varphi(z) -\varepsilon), 
 \quad 
 \varepsilon = \frac{2N^{\delta_0} }{N}.
\end{equation}
\par
Let $\gamma$ be a segment of $E_1$ and $\frac{C\ln N}{N}\leq r \ll 1$, put 
\begin{equation}\label{evI.33}
 \Gamma(r,\gamma) = \{
 z\in\C; ~
 \dist(z,E_1) < r, ~ \Pi(z) \in \gamma \}
\end{equation}
where $\Pi(z)\in E_1$ is the point in $E_1$ with $|\Pi(z)-z)| = \dist(z,E_1)$. 
We want to estimate the number of eigenvalues of $P_{\delta}$ in $\Gamma$. 
(With probability $\geq 1 - \e^{-N^2}$ we know that $P_{\delta}$ has no eigenvalues 
in the exterior region to $E_1$ and we are free to modify $\Gamma$ there. However, 
there seems to be no point to do so in the present situation.)
\\
\par
Choose $z_j^0\in\partial\Gamma$, $r_j=\max \left(
  \frac{1}{2}\dist(z^0_j,\gamma), 4C(\ln N)/N \right)$, $j=1,...,M$
such that:
 \begin{itemize}
\item $\partial \Gamma \subset \cup_{j=1}^M D(z_j^0,r_j/2)$,
\item $\mathrm{dist\,}(z_j^0,\gamma )\ge C(\ln N)/N$ for all $j$,
\item $\# \{j;\, r_j=4C(\ln N)/N \}={\mO}(1)$,  
\item $$
M\le {\mO}\left(\frac{1}{r} \right) +{\mO}(1)
\ln \left(\frac{r}{C(\ln N)/N} \right)
$$
 \end{itemize}
For any choice of $z_j\in D(z_j^0,r_j/4)$, (\ref{evI.32}) holds for
$z=z_1,...,z_M$ with probability
\begin{equation}\label{evI.34}
 \geq 1 - \mO(1)\left(\frac{1}{r} + \ln N\right)N^{2\kappa}\e^{-2N^{\delta_0}}. 
\end{equation}
Applying Theorem 1.2 in \cite{Sj09b}, we get
\begin{equation*}
 \begin{split}
  \bigg|
 \# (\sigma(P_{\delta})\cap \Gamma) &- \frac{N}{2\pi}\int_{\Gamma} \Delta\varphi L(dz)
 \bigg| \\
 &\leq \mO(N) \left(
 \sum_j \varepsilon 
 + \sum_{D(z_j^0,r_j)\cap E_1 \neq \emptyset}
   \int_{D(z_j^0,r_j)}\Delta \varphi L(dz) 
 \right),
 \end{split}
\end{equation*}
where $\varepsilon(z_j) = \frac{2 N^{\delta_0} }{N}$ is given in (\ref{evI.32}). Here 
\begin{equation*}
 \sum_{D(z_j^0,r_j)\cap E_1 \neq \emptyset}
   \int_{D(z_j^0,r_j)}\Delta \varphi L(dz) 
 = \mO(1) \frac{\ln N}{N},
\end{equation*}
and the number of points $z_j^0$ for which $D(z_j^0,r_j)\cap E_1\ne
\emptyset $ is ${\mO}(1)$,
so finally, with probability as in \eqref{evI.34},
\begin{equation}\label{evI.35}
 \begin{split}
  \bigg|
 \# &(\sigma(P_{\delta})\cap \Gamma) - \frac{N}{2\pi}\int_{\Gamma} \Delta\varphi L(dz)
 \bigg| \\
 &\leq \mO(N)\left(\left(\frac{1}{r}+\ln N \right)N^{\delta_0
     -1}+\frac{\ln N}{N} \right)
\le {\mO}(1)N^{\delta_0} \left(\frac{1}{r}+\ln N \right) .
 \end{split}
\end{equation}

\par The measure $\Delta _z\varphi (z)L(dz)$ is invariant under
holomorphic changes of coordinates and in particular, we can replace
$z$ by $\zeta _-$:
\begin{equation}\label{evI.36}
\Delta _z\varphi (z)L(dz)=\Delta _{\zeta _-}\varphi L(d\zeta _-).
\end{equation}
Here we recall that $\varphi $ is given by (\ref{evI.29}) and compute the
right hand side of (\ref{evI.36}). Let $\psi \in C_0^\infty ({\bf
  C}\setminus \{ 0 \})$ be a test function. Then in the sense of
distributions,
\[
\begin{split}
&\int \psi (\zeta _-)\Delta \varphi (\zeta _-)L(d\zeta _-)=\int \Delta
\psi (\zeta _-)\varphi (\zeta _-)L(d\zeta _-)\\
&=\int_{D(0,1)}\Delta \psi (\zeta _-)\varphi (\zeta _-)L(d\zeta
_-)+\int_{{\bf C}\setminus D(0,1)}\Delta \psi (\zeta _-)\varphi (\zeta _-)L(d\zeta _-),
\end{split}
\]
and by Green's formula, this is equal to
$$
\int_{S^1}\psi (\zeta _-)\left( \partial _n \varphi
  _\mathrm{ext}(\zeta _-)-\partial _n\varphi _\mathrm{int}(\zeta _-)
\right) |d\zeta _-|,
$$
where $n$ denotes the exterior unit normal to the unit disc,
$$
\begin{cases}
\varphi _\mathrm{int}=\ln |a| +C/N,\\
\varphi _\mathrm{ext}=\ln |a| +\ln |\zeta _-|+C/N.
\end{cases}
$$
Since $\partial _n\varphi _\mathrm{int}(\zeta _-)=0$, $\partial _n\varphi
_\mathrm{ext}(\zeta _-)=1$ on $S^1$, we get
$$
\int \psi (\zeta _-)\Delta \varphi (\zeta _-) L(d\zeta _-)=\int_{S^1}\psi
(\zeta _-) |d\zeta _-|,
$$
i.e.\
\begin{equation}\label{evI.37}
\Delta \varphi (\zeta _-)L(d\zeta _-)=L_{S^1}(ds),\\ \hbox{the length
  measure on }S^1.
\end{equation} 
Recall that $\Gamma \cap E_1=\gamma $. Letting $\gamma $ also denote
the corresponding arc in $S^1_{\zeta _-}$, depending on the context,
we see that
$$
\int_{\Gamma }\Delta \varphi L(dz)=\int_{\gamma }L_{S^1}(d\zeta
_-)
$$
i.e.\ the length of $\gamma $ with respect to the $\zeta
_-$-coordinates. To make the connection with Weyl's formula, we write
$\zeta _-=e^{i\xi }$, so that $S^1_{\zeta _-}$ corresponds to $\xi \in
{\bf R}/2\pi {\bf Z} $. Then $\gamma \subset S^1$ can be identified with $\{
\xi \in {\bf R}/2\pi {\bf Z};\, e^{i\xi }\in \gamma  \}$. Viewing again $\gamma $ as a segment
in $E_1$, we get
\begin{equation}\label{evI.38}
\int_{\Gamma }\Delta \varphi L(dz)=\mathrm{length\,}\left(\{ \xi \in {\bf
    R}/2\pi {\bf Z};\, P_\mathrm{I}(\xi )\in \gamma \} \right),
\end{equation}
where the right hand side does not change if we replace $\gamma $ by
$\Gamma $
and where $P_\mathrm{I}$ is the symbol given in (\ref{int.5}), or equivalently
\begin{equation}\label{evI.39}
\int_{\Gamma }\Delta \varphi L(dz)=\mathrm{vol}_{S^1_\xi }P_\mathrm{I}^{-1}(\Gamma ).
\end{equation}
If we view $P_\mathrm{I}$ as a function on $]0,N]_x\times S^1_\xi $,
where $]0,N]=\{ x\in {\bf R};\, 0<x\le N \}$, then
\begin{equation}\label{evI.40}
N\int_{\Gamma }\Delta \varphi L(dz)=\mathrm{vol}_{]0,N]\times
  S^1}P_\mathrm{I}^{-1}(\Gamma )
\end{equation}
Using this in \eqref{evI.35}, we get Theorem \ref{evI2}.
%
%
\%bibliography{bibliography}
\providecommand{\bysame}{\leavevmode\hbox to3em{\hrulefill}\thinspace}
\providecommand{\MR}{\relax\ifhmode\unskip\space\fi MR }
\providecommand{\MRhref}[2]{%
  \href{http://www.ams.org/mathscinet-getitem?mr=#1}{#2}
}
\providecommand{\href}[2]{#2}


\begin{thebibliography}{10}

\bibitem{BM}
W.~Bordeaux-Montrieux, \emph{{Loi de Weyl presque s{\^u}re et r{\'e}solvent
  pour des op{\'e}rateurs diff{\'e}rentiels non-autoadjoints, Th{\'e}se}},
  pastel.archives-ouvertes.fr/docs/00/50/12/81/PDF/manuscrit.pdf (2008).

\bibitem{BoSi99}
A.~B\"ottcher and B.~Silbermann, \emph{Introduction to large truncated Toeplitz
  matrices}, Springer, 1999.

\bibitem{ZwChrist10}
T.J. Christiansen and M.~Zworski, \emph{{Probabilistic Weyl Laws for Quantised
  Tori}}, Communications in Mathematical Physics \textbf{299} (2010).

\bibitem{Da97}
E.~B. Davies, \emph{{Pseudospectra of Differential Operators}}, J. Oper. Th
  \textbf{43} (1997), 243--262.

\bibitem{Da99}
E.B. Davies, \emph{{Pseudo{\textendash}spectra, the harmonic oscillator and
  complex resonances}}, Proc. of the Royal Soc.of London A \textbf{455} (1999),
  no.~1982, 585--599.

\bibitem{DaHa09}
E.B. Davies and M.~Hager, \emph{{Perturbations of Jordan matrices}}, J. Approx.
  Theory \textbf{156} (2009), no.~1, 82--94.

\bibitem{NSjZw04}
N.~Dencker, J.~Sj{\"o}strand, and M.~Zworski, \emph{{Pseudospectra of
  semiclassical (pseudo-) differential operators}}, Communications on Pure and
  Applied Mathematics \textbf{57} (2004), no.~3, 384--415.

\bibitem{TrEm05}
M.~Embree and L.~N. Trefethen, \emph{{Spectra and Pseudospectra: The Behavior
  of Nonnormal Matrices and Operators}}, Princeton University Press, 2005.

\bibitem{GoKh00}
I.Y. Goldsheid and B.A. Khoruzhenko, \emph{Eigenvalue curves of asymmetric
  tridiagonal random matrices}, Elec. J. of Probability. \textbf{5} (2000),
  no.~16, 1--28.

\bibitem{GuMaZe14}
A.~Guionnet, P.~Matchett Wood, and {0. Zeitouni}, \emph{{Convergence of the
  spectral measure of non-normal matrices}}, Proc.~AMS \textbf{142} (2014),
  no.~2, 667--679.

\bibitem{Ha06b}
M.~Hager, \emph{{Instabilit{\'e} Spectrale Semiclassique d{\rq}Op{\'e}rateurs
  Non-Autoadjoints II}}, Annales Henri Poincare \textbf{7} (2006), 1035--1064.

\bibitem{Ha06}
\bysame, \emph{{Instabilit{\'e} spectrale semiclassique pour des op{\'e}rateurs
  non-autoadjoints I: un mod{\`e}le}}, Annales de la facult{\'e} des sciences
  de Toulouse S{\'e}. 6 \textbf{15} (2006), no.~2, 243--280.

\bibitem{HaSj08}
M.~Hager and J.~Sj{\"o}strand, \emph{{Eigenvalue asymptotics for randomly
  perturbed non-selfadjoint operators}}, Mathematische Annalen \textbf{342}
  (2008), 177--243.

\bibitem{HaNe96}
N.~Hatano and D.R. Nelson, \emph{Localization transitions in non-hermitian
  quantum mechanics}, Physical Review Letters \textbf{77} (1996), 570--573.

\bibitem{Sj15}
J.~Sj{\"o}strand, \emph{{Non-self-adjoint differential operators, spectral
  asymptotics and random perturbations }}, Book in preparation.

\bibitem{SjAX1002}
\bysame, \emph{{Spectral properties of non-self-adjoint operators}}, Actes des
  Journ{\'e}es d'{\'e}.d.p. d'{\'E}vian (2009).

\bibitem{SjVo15}
J.~Sj{\"o}strand and M.~Vogel, \emph{{Interior eigenvalue density of Jordan
  matrices with random perturbations}},  (2015), accepted for publication as
  part of a book in honour of Mikael Passare in the series Trends in
  Mathematics, Springer/Birkh{\"a}user, e-preprint [arxiv:1412.2230].

\bibitem{Sj09b}
Johannes Sj\"ostrand, \emph{Counting zeros of holomorphic functions of
  exponential growth}, Journal of pseudodifferential operators and applications
  \textbf{1} (2010), no.~1, 75--100.

\bibitem{Tr97}
L.N. Trefethen, \emph{{Pseudospectra of linear operators}}, SIAM Rev.
  \textbf{39} (1997), no.~3, 383--406.

\bibitem{Vo14b}
M.~Vogel, \emph{{Eigenvalue interaction for a class of non-selfadjoint
  operators under random perturbations}},  (2014), e-preprint
  [arxiv:1412.0414].

\bibitem{Vo14}
\bysame, \emph{{The precise shape of the eigenvalue intensity for a class of
  non-selfadjoint operators under random perturbations}},  (2014), submitted,
  e-preprint [arXiv:1401.8134].

\bibitem{Wi94}
H.~Widom, \emph{Eigenvalue distribution for nonselfadjoint Toeplitz matrices},
  Operator Theory: Advances and Applications \textbf{71} (1994), no.~1--8.

\end{thebibliography}
\end{document}